\documentclass[a4paper,11 pt,twoside]{amsart}

\usepackage[utf8]{inputenc}
\usepackage[english]{babel}
\usepackage{times}
\usepackage{amsmath}
\usepackage{amsfonts}
\usepackage{amssymb}
\usepackage{amsmath}
\usepackage{amsthm}
\usepackage{graphicx}
\usepackage{array}
\usepackage{color}
\usepackage{mathrsfs}
\usepackage{hyperref}
\usepackage{esint} 
\usepackage{tikz}
\usepackage{ upgreek }
\usepackage{enumitem}
\usepackage{comment}

\definecolor{darkgreen}{rgb}{0.1,0.7,0.1}
\definecolor{darkred}{rgb}{0.7,0.1,0.1}

\newcommand{\mat}[4]{\left( \begin{array}{cc}
#1 & #2  \\
#3 & #4  \\
\end{array} \right)}

\newtheorem{theorem}{Theorem}
\newtheorem{lemma}{Lemma}[section]
\newtheorem{proposition}[lemma]{Proposition}
\newtheorem{corollary}[lemma]{Corollary}

\newtheorem{remark}[lemma]{Remark}

\newtheorem*{theorem*}{Theorem}

\newcommand{\arccotan}{\mathrm{arccotan\,}}

\newcommand{\eps}{\varepsilon}

\newcommand\symb[2][\bf]{{\mathchoice{\hbox{#1#2}}{\hbox{#1#2}}%
        {\hbox{\scriptsize#1#2}}{\hbox{\tiny#1#2}}}}

\def\R{{\symb R}}

\def\Z{{\symb Z}}
\def\C{{\symb C}}

\def\P{{\symb P}}

\def\Sch{\mbox{Sch}}
\def\tSch{\mathtt{CS}}

\def\tSchE{\cS^{(E)}_L}

\def\un{\mathbf{1}}

\renewcommand{\P}{\mathbb{P}}
\newcommand{\E}{\mathbb{E}}

\newcommand{\bbQ}{\mathbb{Q}}

\newcommand{\cB}{\mathcal{B}}
\newcommand{\cC}{\mathcal{C}}
\newcommand{\cD}{\mathcal{D}}
\newcommand{\cE}{\mathcal{E}}
\newcommand{\cF}{\mathcal{F}}

\newcommand{\cH}{\mathcal{H}}

\newcommand{\cN}{\mathcal{N}}

\newcommand{\cP}{\mathcal{P}}

\newcommand{\cS}{\mathcal{S}}
\newcommand{\cT}{\mathcal{T}}

\newcommand{\cW}{\mathcal{W}}

\newcommand{\gl}{\lambda}

\newcommand{\bs}[1]{\boldsymbol{#1}}
\newcommand{\HH}{\bs{\mathcal{H}}}

\hypersetup{
citecolor=blue,
colorlinks=true, 
breaklinks=true, 
urlcolor= blue, 
linkcolor= black, 
bookmarksopen=true, 
pdftitle={Anderson Hamiltonian}, 
}

\begin{document}

\title[Delocalized phase of the Anderson Hamiltonian]{The delocalized phase\\of the Anderson Hamiltonian in $1$-d}

\author{Laure Dumaz}
\address{
CNRS \& Department of Mathematics and Applications, \'Ecole Normale Sup\'erieure (Paris), 45 rue d’Ulm, 75005 Paris, France}
\email{laure.dumaz@ens.fr}

\author{Cyril Labb\'e}
\address{
Universit\'e Paris-Dauphine, PSL University, UMR 7534, CNRS, CEREMADE, 75016 Paris, France\\
\& \'Ecole Normale Sup\'erieure, UMR 8553, CNRS, DMA, 75005 Paris, France}
\email{labbe@ceremade.dauphine.fr}

\vspace{2mm}

\date{\today}

\maketitle

\begin{abstract}

We introduce a random differential operator, that we call $\tSch_\tau$ operator, whose spectrum is given by the $\Sch_\tau$ point process introduced by Kritchevski, Valk\'o and Vir\'ag~\cite{KVV} and whose eigenvectors match with the description provided by Rifkind and Vir\'ag in~\cite{RV}. This operator acts on $\R^2$-valued functions from the interval $[0,1]$ and takes the form:
\begin{align*}
2 \mat{0}{-\partial_t}{\partial_t}{0} + \sqrt{\tau} \begin{pmatrix} d\cB + \frac1{\sqrt 2} d\cW_1 & \frac1{\sqrt 2} d\cW_2\\ \frac1{\sqrt 2} d\cW_2 & d\cB - \frac1{\sqrt 2} d\cW_1\end{pmatrix}\,,
\end{align*}
where $d\cB$, $d\cW_1$ and $d\cW_2$ are independent white noises.
Then, we investigate the high part of the spectrum of the Anderson Hamiltonian $\mathcal{H}_L := -\partial_t^2 + dB$ on the segment $[0,L]$ with white noise potential $dB$, when $L\to\infty$. We show that the operator $\cH_L$, recentred around energy levels $E \sim L/\tau$ and unitarily transformed, converges in law as $L\to\infty$ to $\tSch_\tau$ in an appropriate sense. This allows to answer a conjecture of Rifkind and Vir\'ag~\cite{RV} on the behavior of the eigenvectors of $\cH_L$. Our approach also explains how such an operator arises in the limit of $\cH_L$. Finally we show that at higher energy levels, the Anderson Hamiltonian matches (asymptotically in $L$) with the unperturbed Laplacian $-\partial_t^2$. In our companion paper~\cite{DL3} it is shown that at energy levels much smaller than $L$, the spectrum is localized with Poisson statistics: the present paper therefore identifies the delocalized phase of the Anderson Hamiltonian.
\medskip

\end{abstract}

\setcounter{tocdepth}{1}
\tableofcontents

\section{Introduction}

The original motivation of the present article was to study the asymptotic behavior as $L\to\infty$ of the high part of the spectrum of the Anderson Hamiltonian
$$ \cH_L := -\partial_t^2 + dB\;,\quad t\in(0,L)\;,$$
endowed with Dirichlet b.c., where $B$ is standard Brownian motion. It is known~\cite{Fukushima} that $\cH_L$ is a self-adjoint operator with discrete spectrum, bounded below and of multiplicity one. In the sequel we denote by $(\lambda_k)_{k\ge 1}$ the increasing sequence of eigenvalues of $\cH_L$ and by $(\varphi_k)_{k\ge 1}$ the associated sequence of eigenvectors normalized in $L^2(0,L)$.\\

By high part, we mean eigenvalues of order $L$ or more. More precisely, we aimed at understanding the local statistics of $\cH_L$ near some energy level $E=E(L)$ that goes to $\infty$ with at least linear speed in $L$. This question falls in the topic of Anderson localization for Schr\"odinger operators: it is expected that, the higher $E$ lies within the spectrum the less localized the eigenvectors are. In the companion papers~\cite{DL1,DL3}, we showed that around energies $E \ll L$, the local statistics of the eigenvalues converge to a Poisson point process and the eigenvectors are localized.\\

In the present article, we will show that the rest of the spectrum is delocalized. However two distinct regimes will arise:\begin{enumerate}
\item Critical: $E \sim L/\tau$ for some $\tau \in (0,\infty)$,
\item Top: $E \gg L$.
\end{enumerate}
The most interesting case will be the Critical regime, so we now focus on it. At the end of the introduction, we will present our results in the other regime.

\bigskip

There is a formal analogy between our Critical regime, and the behavior of critical models of $N\times N$ tridiagonal matrices of the form ``discrete Laplacian + diagonal noise'' as studied by Kritchevski, Valk\'o and Vir\'ag in~\cite{KVV}. It was shown therein that the spectrum of these models, after appropriate recentring around some energy level that depends on a parameter $\tau \in (0,\infty)$, converges in law as $N\to\infty$ to a random point process that they called $\Sch_\tau$. The denomination ``critical'' for the underlying matrix models comes from the following fact: the family of point processes $\Sch_\tau$ that arises in the limit interpolates between the Poisson point process ($\tau =\infty$, see \cite{AllezDumazSine} for a proof for a related model) and the picket fence $2\pi \Z$ ($\tau=0$).\\

The point process $\Sch_\tau$, introduced in \cite{KVV}, admits a nice characterization in terms of a coupled system of SDEs. Let $\cB$ and $\cW$ be independent real and complex\footnote{By complex Brownian motion, we mean that the real and imaginary parts are independent standard Brownian motions.} Brownian motions, and consider
\begin{align}
d\Theta_\gl(t) = \frac{\gl}{2} dt + \frac{\sqrt{\tau}}{2} d\cB(t) + \frac{\sqrt{\tau}}{2\sqrt 2} \Re(e^{2i \Theta_\gl(t)} d\cW(t))\;, t\in [0,1]\;, \gl \in \R\;, \label{defThetagl}
\end{align}
where $\Theta_\gl(0):=0$ (note that this family $(\Theta_\gl)_\gl$ depends on $\tau$). The point process $\Sch_\tau$ can be defined as
\begin{align}
\Sch_{\tau} := \{ \lambda \in \R: \Theta_{\gl}(1) \in \pi \Z\}\;.\label{Characdiffusion:Schtau}
\end{align}
\medskip

Our first result shows not only convergence of the local statistics of $\cH_L$ near $E\sim L/\tau$ towards $\Sch_\tau$, but also the joint convergence of the eigenvalues and eigenvectors (rescaled into probability measures) to an explicit $2d$-point process whose first component is $\Sch_\tau$. To state the result, let us introduce
\begin{align}
 d \rho_{\lambda}(t) = \frac{\tau}{8} dt + \frac{\sqrt{\tau}}{2\sqrt 2} \Im(e^{2i \Theta_\gl(t)} d\cW(t))\;,\quad t\in [0,1]\;,\quad \lambda \in \R\;, \label{defrhogl}
\end{align}
with $\rho_\lambda(0) = 0$. Let us also set $\ell_E := \{L\sqrt E\}_\pi$ where $\{x\}_\pi$ is the unique value in $[0,\pi)$ such that $x= \{x\}_\pi$ modulo $\pi$.

\begin{theorem}\label{Th:Joint1}
Fix $\tau > 0$ and assume that $E=E(L) \sim L/\tau$. As $L\to\infty$, the random point process
\begin{align*}
\Big\{\Big((L/\sqrt{E})(\lambda - E) + 2\ell_E, \; L \,\varphi^2_\gl(L t) dt\Big), \;\gl \mbox{ is an eigenvalue of }\mathcal{H}_L \Big\}
\end{align*}
converges in law to
\begin{align*}
\Big\{\Big(\gl, \frac{e^{2\rho_{\gl}(t)} dt}{\int_{[0,1]} e^{2\rho_{\gl}(t)} dt}  \Big), \;\gl \in \Sch_\tau \Big\}\;.
\end{align*}
\end{theorem}
\noindent In this statement, point processes are seen as elements of the set of locally finite measures on $\R\times \cP([0,1])$ endowed with the vague topology, that is, the topology that makes continuous $\mu \mapsto \langle \mu,f\rangle$ for all bounded and continuous maps $f:\R\times\cP([0,1])\to\R$ that are compactly supported in their first coordinate.

\bigskip

At this point, two natural questions arise:\begin{enumerate}
\item Does there exist a self-adjoint operator whose spectrum is given by $\Sch_\tau$ and whose eigenvectors are (related to) the processes $e^{2\rho_{\gl}}$ ?
\item Does the above convergence hold at the level of operators ?
\end{enumerate}
\bigskip

We positively answer to both questions. Let us mention here that Valk\'o and Vir\'ag have constructed in~\cite{VVoperator} a general framework of random operators, corresponding to the limit of several famous random matrices models and notably the bulk, soft and hard edges of $\beta$-ensembles. In particular, they have introduced an operator whose spectrum corresponds to the translation invariant version of $\Sch_\tau$. Those operators are defined through the theory of \emph{canonical systems} of De Branges \cite{DeBranges}.

In the present paper, we also rely on the theory of canonical systems to construct an operator whose spectrum is given by $\Sch_\tau$. However, our operator takes a different form from those appearing in~\cite{VVoperator}. This is due to our additional constraint on the eigenvectors (they must be related to $e^{2\rho_\gl}$) that is absent in~\cite{VVoperator}: recall that we impose ourselves this additional constraint in order for our operator to arise as a scaling limit of $\cH_L$. We refer to Remark \ref{Rk:VV} for more comments on the difference with~\cite{VVoperator}. We also believe that our approach sheds some light on the class of operators that should appear as scaling limits of tridiagonal matrices or continuum random Schr\"odinger operators.


\subsection*{The Critical Schr\"odinger operator $\tSch_\tau$}

Let us start by introducing the limiting operator and its first properties. We (formally) define the following operator on $L^2([0,1],\R^2)$, that we call $\tSch_\tau$ for Critical Schr\"odinger:
\begin{equation}
\tSch_\tau := 2 \begin{pmatrix}
0 & -\partial_t\\
\partial_t & 0
\end{pmatrix}+ \sqrt{\tau} \begin{pmatrix} d\cB + \frac1{\sqrt 2} d\cW_1 & \frac1{\sqrt 2} d\cW_2\\ \frac1{\sqrt 2} d\cW_2 & d\cB - \frac1{\sqrt 2} d\cW_1\end{pmatrix}\;, \label{DefopSch}
\end{equation}
where $\cB,\cW_1,\cW_2$ are independent Brownian motions. The precise definition will be given in Section \ref{sec:limitingoperator}, let us only mention that we endow this operator with Dirichlet b.c., that is, any function $f\in L^2([0,1],\R^2)$ lying in the domain is such that $f(0)$ and $f(1)$ are parallel to $(0,1)^\intercal$, where we denote by $M^\intercal$ the transpose of any matrix $M$. Let us emphasize that the operator $\tSch_\tau$ acts on a space of $\R^2$-valued functions while our initial operator $\cH_L$ acts on $\R$-valued functions: the reason for the ``enlargement'' of the underlying space will appear below. Note also that $\tSch_\tau$ is of the form
\begin{align*}
 2 \mat{0}{-\partial_t}{\partial_t}{0} + \mbox{ ``noise matrix''}\,,
\end{align*}
a form that was conjectured by Edelman-Sutton~\cite{EdelmanSutton} for the limit of certain tridiagonal ensembles.

\begin{theorem}\label{Th:Sch}
The operator $\tSch_\tau$ is self-adjoint with discrete spectrum. Its collection of eigenvalues and normalized eigenvectors coincides in law with
$$\Big\{ (\lambda, \Psi_{\lambda}) : \Theta_{\lambda}(1) \in \pi \Z \Big\}\;,$$
where $\Psi_\lambda :=  \frac{e^{\rho_{\lambda}}}{\| e^{\rho_\lambda} \|_{L^2([0,1])}} \begin{pmatrix} \sin \Theta_\lambda \\ \cos \Theta_\lambda \end{pmatrix}$, where $\Theta_\gl$ and $\rho_\gl$ follow the diffusions \eqref{defThetagl} and \eqref{defrhogl}.
\end{theorem}

\medskip

Before we address our second question, let us compute the intensity measure of the point process of eigenvalues / eigenvectors, that is, the measure on $\R\times\cC([0,1],\R^2)$ defined by
$$ \mu_\tau(A) := \E\Big[\sum_{\gl: \Theta_{\gl}(1) \in \pi \Z} \un_A\big(\gl,\Psi_{\gl}(\cdot)\big)\Big]\;.$$

Let $g_{m,\sigma^2}$ be the density of the real gaussian law $\cN(m,\sigma^2)$.

\begin{theorem}\label{Th:Intensity}
For any non-negative measurable map $G$ on $\R\times \cC([0,1],\R^2)$ we have
\begin{align*}
\int G(\gl, \Psi) d\mu_\tau(\gl, \Psi)= \int \sum_{n \in \Z} \E\Big[G\big(\gl, X^{(n)}\big)\Big] g_{\lambda, \frac32 \tau}(2n\pi) d\gl\;,
\end{align*}
where
$$ X^{(n)}(t) = \frac{e^{\frac{\sqrt \tau}{2\sqrt{2}}\mathcal{B}_1(t-U) - \frac{\tau}8 |t-U| }}{\big\| e^{\frac{\sqrt \tau}{2\sqrt{2}}\mathcal{B}_1(\cdot-U) - \frac{\tau}8 |\cdot - U| } \big\|_{L^2([0,1])}}\begin{pmatrix} \sin \beta^{(n\pi)} (t) \\\cos \beta^{(n\pi)} (t) \end{pmatrix}\;,\quad t\in [0,1]\;,$$
and\begin{itemize}
\item $\mathcal{B}_1$ is a two-sided, real Brownian motion, 
\item $\beta^{(n\pi)}$ is a scaled Brownian bridge between $(0,0)$ and $(1,n\pi)$:
$$ \beta^{(n\pi)}(t)= \frac{1}{2} \sqrt{\frac{3\tau}{2}}(\mathcal{B}_2(t) - t\mathcal{B}_2(1)) + n \pi t \;,\quad t\in [0,1]\;,$$
associated to an independent Brownian motion $\cB_2$,
\item $U$ is a uniform random variable on $[0,1]$ independent of $\mathcal{B}_1$ and $\mathcal{B}_2$.
\end{itemize}
\end{theorem}

As an immediate corollary of this result we recover the density of the intensity measure of $\Sch_\tau$, already obtained in Theorem 10 of \cite{KVV} with a different method:
\begin{align*}
\E\Big[\sum_{\gl \in \mbox{Sch}_{\tau}} G(\gl)\Big] &= \int_\gl  G(\gl) \sum_{n \in \Z} g_{\lambda, \frac32 \tau}(2n\pi) d\gl\,.
\end{align*}
Note that the above density is not translation invariant: it is only $2\pi$-periodic. As $\tau\downarrow 0$, the intensity measure converges to $\sum_{n\in\Z} \delta_{2n\pi}$, the intensity measure of the so-called picket fence. On the other hand, as $\tau \uparrow \infty$, it ``converges'' to an infinite uniform measure on $\R$. This is another hint that $\Sch_\tau$ is critical, in the sense that it interpolates between the localized and delocalized phases of Schr\"odinger operators.

\subsection*{Convergence of the eigenvectors}

We now address our second question on the convergence at the operator level. Given the statement of Theorem \ref{Th:Joint1}, one naturally starts from the recentered operator $(L/\sqrt E)(\cH_L - E) + 2\ell_E$.\\
For convergence purposes, one would like to deal with functions on $(0,1)$ instead of $(0,L)$. Therefore we conjugate this operator with the rescaling map $g \mapsto g(L\; \cdot )$ from $(0,L)$ to $(0,1)$ and this yields\footnote{The corresponding conjugation is not unitary, unless one defines ${\cH}^{(E)}$ on the Hilbert space $L^2((0,1),L dt)$.} the operator ${\cH}^{(E)}$ defined through
\begin{align}
{\cH}^{(E)} f := -\frac1{L\sqrt E} f'' + \sqrt{\frac{L}{E}} dB^{(L)} f + (2\ell_E - L\sqrt E) f\;,\quad t\in [0,1]\;,\label{defopHE}
\end{align}
where $B^{(L)}(t) = L^{-1/2} B(t L)$ is again a standard Brownian motion. The eigenvalues  and (normalized) eigenvectors $(\gl, {\varphi}_\gl^{(E)})$ of ${\cH}^{(E)}$ are in one-to-one correspondence with those of $\cH_L$ via
\begin{equation}\label{Eq:varphimu}
{\varphi}^{(E)}_\gl = \sqrt L \,\varphi_\mu(L\; \cdot)\;,\quad \gl = \frac{L}{\sqrt E}( \mu - E) + 2\ell_E\;.
\end{equation}
Note that the domain of ${\cH}^{(E)}$ is a subset of $L^2((0,1),dt)$ which is nothing but the image of the domain of $\cH_L$ through the rescaling map.

\bigskip

Looking back at the statement of Theorem \ref{Th:Joint1}, we observe that the point process that appears in the limit is nothing but the following projection of the eigenvalues / eigenvectors of $\tSch_\tau$
$$\bigg\{ \Big(\lambda, {|\Psi_{\lambda}|^2}\Big) : \Theta_{\lambda}(1) \in \pi \Z \bigg\}\;.$$
A first naive guess would then be that the operator ${\cH}^{(E)}$ converges to $\tSch_\tau$, and this would formally imply that the eigenvectors ${\varphi}^{(E)}_\gl$ converge to the eigenvectors of $\tSch_\tau$. It turns out that in this regime of energy, the eigenvectors ${\varphi}^{(E)}_\gl$ oscillate too much to converge as functions: indeed, from standard arguments of the theory of Sturm-Liouville operators, one can deduce that their numbers of zeros on $(0,1)$ is of order $L\sqrt E$.\\


One actually needs to remove these oscillations for the eigenvectors to converge. This can be done by considering the associated probability measures as we did in Theorem \ref{Th:Joint1}. However, in order to get convergence at the operator level, we need to remove these oscillations at the level of the \emph{functions}. To do so, we successively apply two transformations:

\medskip

\emph{Step 1: From $\R$ to $\R^2$.} We consider the pair formed by the eigenvector and its derivative:
$$ \begin{pmatrix} \varphi_\gl^{(E)} \\ \frac1{L\sqrt E} (\varphi_\gl^{(E)})'\end{pmatrix}\;.$$ 

\emph{Step 2: Unrotate.}
Set\footnote{While $L\sqrt E$ is nothing but the order of magnitude of the oscillations of ${\varphi}_\mu^{(E)}$, the correction $\ell_E$, which is of order $1$, is more subtle: it is chosen in such a way that the rotation preserves the Dirichlet b.c.} $E' := L\sqrt E -\ell_E$ and introduce the (evolving) rotation matrix
\begin{align}
C_{E'} = C_{E'}(t) := \begin{pmatrix}  \cos E' t & - \sin E' t \\ \sin E' t & \cos E' t\end{pmatrix}\;.\label{def:CE'}
\end{align}
\indent Then, we define
\begin{equation}\label{Eq:PsiE} \Psi_\gl^{(E)} := C_{E'} \begin{pmatrix} \varphi_\gl^{(E)} \\ \frac1{L\sqrt E} (\varphi_\gl^{(E)})'\end{pmatrix}\;.\end{equation}

\begin{theorem}[Joint convergence of the eigenvalues and eigenvectors]\label{Th:Joint2}
Fix $\tau > 0$ and consider $E=E(L) \sim L/\tau$. As $L\to\infty$, the point process on $\R\times \cC([0,1],\R^2)$:
\begin{align*}
\Big\{\Big(\gl,\frac{\Psi^{(E)}_\lambda}{\|\Psi^{(E)}_\lambda\|_{L^2([0,1],\R^2)}}\Big), \;\gl \mbox{ eigenvalue of }\cH^{(E)} \Big\}
\end{align*}
converges in law to the point process of eigenvalues/eigenvectors of $\tSch_{\tau}$, i.e.
\begin{align*}
\{(\gl, \Psi_{\lambda}), \;\gl \mbox{ eigenvalue of }\tSch_\tau\}\;.
\end{align*} 
\end{theorem}
\noindent In this statement, the point processes are seen as elements of  the set of measures on $\R\times \cC([0,1],\R^2)$ that are finite on $K\times \cC([0,1],\R^2)$ for any compact set $K\subset \R$, endowed with the smallest topology that makes continuous $\mu \mapsto \langle \mu,f\rangle$ for all bounded and continuous maps $f:\R\times\cC([0,1],\R^2)\to\R$ that are compactly supported in their first coordinate.\\

This result, combined with our description of the intensity measure of $\tSch_\tau$ given in Theorem \ref{Th:Intensity}, proves part (4) of~\cite[Conjecture 1.3]{RV} on the universal shape of a \emph{typical eigenvector} associated to a ``high eigenvalue'' of $\cH_L$. It is interesting to note that such a behavior can already be observed for the eigenvectors in the localized regime of $\cH_L$ but for high enough energies (energies of order $1 \ll E \ll L$), see \cite{DL3} for more details. It is conjectured in~\cite{RV} that this shape should appear for various critical operators thus its denomination ``universal''. It was also proved to arise in another random Schr\"odinger model recently, see \cite{Nakano2019}.

\subsection*{Convergence at the operator level}

Theorem \ref{Th:Joint2} establishes a relationship between the eigenvalues/eigenvectors of $\tSch_\tau$ and those of $\cH_L$. However it does not exactly answer our second question, and more importantly, it does not explain how the form \eqref{DefopSch} taken by $\tSch_\tau$ arises from $\cH_L$. This is the purpose of our next result.\\

Given Theorem \ref{Th:Joint2}, our second question can be rephrased as follows: is there an operator associated to the point process $(\gl,\Psi_\gl^{(E)})$ and does this operator converge to $\tSch_\tau$ ?\\

We will see later on that a.s.~the space generated by the family of functions $\{\Psi^{(E)}_\gl\}_\gl$ is \emph{not} dense in $L^2([0,1],\R^2)$, and therefore there is \emph{no} self-adjoint operator on $L^2([0,1],\R^2)$ whose eigenvalues/eigenvectors are given by $\{\gl,\Psi^{(E)}_\gl\}_\gl$. Consequently, ``the'' operator that we are looking for must live on a smaller space: this is not surprising since our original operator lives on $L^2([0,1],\R)$.\\

We will construct an operator denoted $\tSchE$, which is unitarily equivalent to $\cH^{(E)}$ and lives on a \emph{quotient} space of $L^2([0,1],\R^2)$. The corresponding unitary map is the lift at the operator level of Steps 1 and 2 above. We will see along the construction of this map how the form \eqref{DefopSch} taken by $\tSch_\tau$ arises from $\cH^{(E)}$, see in particular Equation \eqref{Eq:yE} below. We postpone to the end of the introduction the detailed presentation of this construction since it requires some notation.

We now state our convergence result at the operator level: since there are some issues with the underlying spaces on which our operator acts, the precise statement will be given in Subsection \ref{Subsec:PreciseResolv}.

\begin{theorem}[Strong-resolvent convergence]\label{Th:CVoperator}
Fix $\tau > 0$ and consider $E=E(L) \sim L/\tau$. As $L\to\infty$, the operator $\tSchE$ converges in law, in the strong resolvent sense, towards the operator $\tSch_\tau$. However it does not converge in law to $\tSch_\tau$ in the norm resolvent sense.
\end{theorem}

This answers our second question. Note that the strong resolvent convergence does not imply the convergence of the eigenvalues/eigenvectors, see e.g.~\cite{WeidmannResolvent}: in particular, Theorem \ref{Th:CVoperator} does not imply Theorem \ref{Th:Joint2}.

\subsection*{Top of the spectrum}

Finally, let us examine the spectrum of $\cH_L$ for energies $E$ that go to $\infty$ much faster than $L$. Note that now $L/E \to 0$ so that, formally, we are in the same situation as before but with $\tau =0$. We keep the same definitions for $\cH^{(E)}$, $\Psi^{(E)}_\gl$ and $\tSchE$ in that case (see the next paragraph for the precise definition of the latter). Our next result shows that all the previous results still hold but the limits are now deterministic: at first order, the influence of the white noise becomes negligible.

\begin{theorem}\label{Th:Top}
Consider $E=E(L) \gg L$. As $L\to\infty$, the operator $\tSchE$ converges in probability, in the strong resolvent sense, towards the picket fence operator on $L^2_I$
$$ \mathtt{F} := 2 \begin{pmatrix} 0 & -\partial_t\\ \partial_t & 0 \end{pmatrix}\;,$$
endowed with Dirichlet b.c. However the convergence does not hold in the norm resolvent sense.\\
Furthermore the point process:
\begin{align*}
\Big\{\Big(\gl, \frac{\Psi^{(E)}_\lambda}{\|\Psi^{(E)}_\lambda\|_{L^2([0,1],\R^2)}}\Big), \;\gl \mbox{ eigenvalue of }\cH^{(E)} \Big\}
\end{align*}
converges in probability to the eigenvalues/eigenvectors of $\mathtt{F}$, which happen to be given by
\begin{align*}
\Big\{\Big(\lambda, \begin{pmatrix} \cos(\frac{\lambda}{2} \cdot ) \\[0.3cm] \sin (\frac{\lambda}{2} \cdot ) \end{pmatrix} \Big), \;\gl \in 2\pi \Z \Big\}\;.
\end{align*} 
\end{theorem}

\subsection*{A non-trivial unitary map}

We now present the unitary transformation that transforms the operator $\cH_L$ into $\tSchE$. For the sake of clarity, it is first spelled out at the level of the SDEs solved by the eigenvectors, and then at the operator level. Let us write the family of SDEs associated to the operator $\cH^{(E)}$:
\begin{align}\label{Eq:uE}
-\frac1{L\sqrt E} (u_\gl^{(E)})'' + \sqrt{\frac{L}{E}}u_\gl^{(E)} dB^{(L)} + (2\ell_E -  L\sqrt E) u_\gl^{(E)} = \gl  u_\gl^{(E)} \;, \quad t\in [0,1]\;,
\end{align}
with initial conditions $u_\gl^{(E)}(0) = 0$ and\footnote{Any non-zero value for $(u_\gl^{(E)})'(0)$ would do, we choose $L\sqrt E$ for later convenience.} $(u_\gl^{(E)})'(0) = L\sqrt E$. By the Sturm-Liouville theory, the parameter $\lambda$ is an eigenvalue of $\cH^{(E)}$ if and only if $(u_\gl^{(E)}(1),(u_\gl^{(E)})'(1))$ is parallel to $(0,1)$; and in that case, the associated normalized eigenvector $\varphi_\gl^{(E)}$ is a multiple of $u_\gl^{(E)}$.\\

Our transformation of the operator $\cH^{(E)}$ can be factorized into two steps:\\

\emph{Step 1: From $\R$ to $\R^2$.} Consider the matrix 
$$T = \begin{pmatrix} 1 & 0\\0&0\end{pmatrix}\;.$$
The above collection of SDEs \eqref{Eq:uE} can be rewritten
\begin{equation}\label{SDEs:u}
\begin{pmatrix} \sqrt{\frac{L}{E}} \,dB^{(L)} + 2\ell_E - L\sqrt E & - \partial_t \\ \partial_t & -L\sqrt E\end{pmatrix} \begin{pmatrix} u_\gl^{(E)} \\ \frac1{L\sqrt E} (u_\gl^{(E)})' \end{pmatrix} = \gl\; T  \begin{pmatrix} u_\gl^{(E)} \\ \frac1{L\sqrt E} (u_\gl^{(E)})' \end{pmatrix}\;.
\end{equation}
The l.h.s.~takes a form similar to \eqref{DefopSch}: the main difference consists in the unbounded terms $L\sqrt E$ that still need to be ``killed'', this will be the purpose of Step 2. Before we present it, let us now present the transformation corresponding to Step 1 at the operator level.

\smallskip

Although we lifted our system from $\R$ to $\R^2$, it is \emph{not} canonically associated to a densely defined operator on $L^2((0,1),\R^2)$ and one needs to work on a ``smaller'' space. More precisely, let $L^2_T((0,1),\R^2)$ be the Hilbert space of all measurable functions $f:(0,1)\to\R^2$ such that $\int_0^1 f^\intercal T f < \infty$. In other words, any element of $L^2_T$ can be seen as an equivalent class of $L^2((0,1),\R^2)$ for the relation $f \sim g \Longleftrightarrow f_1 = g_1$ a.e. where $f = (f_1,f_2)^\intercal$ and $g = (g_1,g_2)^\intercal$.

We then define the unitary map $\iota : L^2((0,1),\R) \to L^2_T((0,1),\R^2)$ that associates to any $f \in L^2((0,1),\R)$ its canonical equivalent class in $L^2_T((0,1),\R^2)$, and set
$$ \HH^{(E)} := \iota {\cH}^{(E)} \iota^{-1}\;.$$
This transformation is explained in more details in Section \ref{sec:CVoperator}, where we will also prove that the eigenvalue problem of $\HH^{(E)}$ is associated to the system \eqref{SDEs:u}.

\medskip
\emph{Step 2: Unrotate.} 
We define
$$ y_\gl^{(E)} := C_{E'}  \begin{pmatrix} u_\gl^{(E)} \\ \frac1{L\sqrt E} (u_\gl^{(E)})' \end{pmatrix} \;,$$
where the rotating matrix $C_{E'}$ was defined in \eqref{def:CE'}. A simple computation shows that these processes solve
\begin{equation}\label{Eq:yE}\begin{split}
&\bigg(\begin{pmatrix} 0 & -\partial_t\\ \partial_t & 0\end{pmatrix} + \frac12 {\sqrt{\frac{L}{E}}} \begin{pmatrix} dB^{(L)} + \frac1{\sqrt 2} dW_1^{(L)} & \frac1{\sqrt 2} dW_2^{(L)}\\ \frac1{\sqrt 2} dW_2^{(L)} & dB^{(L)} - \frac1{\sqrt 2} dW_1^{(L)}\end{pmatrix} + \ell_E (2R_{E'}-I)\bigg) y_\gl^{(E)}\\
&= \gl R_{E'}y_\gl^{(E)} \;,\end{split}\end{equation}
with initial condition $y_\gl^{(E)}(0) = (0,1)^\intercal$. Here $I$ is the identity matrix, $R_{E'} := C_{E'} T C_{E'}^\intercal$, and we have introduced the Brownian motions
\begin{align} W_1^{(L)}(t) := \sqrt{2}  \int_0^t \cos (2E's) dB^{(L)}(s)\;,\quad W_2^{(L)}(t) := \sqrt{2}  \int_0^t \sin( 2E's) dB^{(L)}(s)\;.\label{def:W1W2}\end{align}

\smallskip

Let us comment on Equation \eqref{Eq:yE}. First, the conjugation by the rotation matrix $C_{E'}$ removed the unbounded terms $L\sqrt{E}$ from \eqref{SDEs:u} so that all the terms appearing in this equation are bounded w.r.t.~$L$. Moreover, we now see a clear resemblance between \eqref{DefopSch} and the l.h.s.~of \eqref{Eq:yE}. 

Interestingly, although we started from a single Brownian motion, the unbounded oscillations produce two additional, independent Brownian motions in the scaling limit: a phenomenon already observed (at a larger scale) by Valk\'o and Vir\'ag in~\cite{VVlongboxes} in the context of discrete Schr\"odinger operators on long boxes, that heuristically corresponds to dimension $1+$, and that they called \emph{noise explosion} (notice that in both cases, it leads to the delocalization of the spectrum). Second, by the Riemann-Lebesgue Lemma, the matrix $R_{E'}$ converges to $(1/2)I$ and thus the term whose prefactor is $\ell_E$ vanishes in the limit. Finally the prefactor $2$ that appears in \eqref{DefopSch} is actually related to the r.h.s.~of \eqref{Eq:yE} where the term $R_{E'}$ converges towards $(1/2)I$.

\bigskip

Given the formalism already introduced, the desired unitary transformation takes a simple form. Viewing $C_{E'}$ as a unitary map from $L^2_T((0,1),\R^2)$ into $L^2_{R_{E'}}((0,1),\R^2)$, we define
$$ \tSchE := C_{E'} \HH^{(E)} C_{E'}^{-1} = C_{E'} \iota \cH^{(E)} \iota^{-1} C_{E'}^{-1}\;.$$
Consequently, the operator $\tSchE$ is the conjugate of ${\cH}^{(E)}$ with the unitary map $C_{E'} \iota$: it is a self-adjoint operator and its collection $(\gl, \psi^{(E)}_\gl)$ of eigenvalues and normalized eigenfunctions is given by $(\lambda, C_{E'} \iota \,\varphi_\gl^{(E)})$.

The drawback of the present operator formalism is that any element of the domain of $\tSchE$, and therefore any eigenvector, is an element of $L^2_{R_{E'}}$, that is, an equivalent class of $L^2_I$ while the eigenvectors of $\tSch_\tau$ are standard elements of $L^2_I$.
However, it stems from our construction that there is a canonical choice of representative in $L^2_I$ for the elements of the domain. Namely if one sets $P: f\mapsto (f_1,\frac1{L\sqrt E} f_1')^\intercal$ which is densely defined on $L^2_T((0,1),\R^2)$, together with its conjugate $P_{E'} = C_{E'} P C_{E'}^{-1}$, then we will show in Lemma \ref{Lemma:Unitary} that
\begin{equation}\label{Eq:EigenUnitary} P_{E'} \psi^{(E)}_\lambda = \Psi^{(E)}_\gl\;.\end{equation}

%
%

\bigskip

The structure of the rest of the article is as follows. In Section \ref{sec:limitingoperator}, we properly define the Critical Schr\"odinger operator $\tSch_\tau$ using the theory of canonical systems. Then we characterize its spectrum and compute its intensity measure, thus proving Theorems \ref{Th:Sch} and \ref{Th:Intensity}. In Section \ref{sec:CVoperator}, we provide more details on the construction of the operator $\tSchE$ and prove some claims made in the introduction. Then, we show the convergence at the operator level stated in Theorem \ref{Th:CVoperator}. In Section \ref{Sec:SDEs}, we exploit the systems of SDEs associated to the eigenvalues / eigenvectors of the operators at stake, and prove Theorems \ref{Th:Joint1} and \ref{Th:Joint2}. We also prove a technical result stated in Section \ref{sec:CVoperator}. Finally, in Section \ref{sec:top}, we adapt the previous arguments in order to cover the top of the spectrum as in Theorem \ref{Th:Top}.

\section{The Critical Schr\"odinger operator}\label{sec:limitingoperator}

The main objective of this section is to give a rigorous meaning to the operator $\tSch_\tau$ on $L^2_I([0,1],\R^2)$ formally defined by
\begin{equation*}
\tSch_\tau := 2 \begin{pmatrix}
0 & -\partial_t\\
\partial_t & 0
\end{pmatrix}+ \sqrt{\tau} \begin{pmatrix} d\cB + \frac1{\sqrt 2} d\cW_1 & \frac1{\sqrt 2} d\cW_2\\ \frac1{\sqrt 2} d\cW_2 & d\cB - \frac1{\sqrt 2} d\cW_1\end{pmatrix}\;,
\end{equation*}
and endowed with Dirichlet b.c., that is, any $f$ in the domain is such that $f(0)$ and $f(1)$ are parallel to $\begin{pmatrix} 0\\ 1\end{pmatrix}$. Here $\cB$, $\cW_1$ and $\cW_2$ are independent Brownian motions.\\

To carry out the construction, we exhibit a transformation that maps the above formal operator onto a random differential operator whose construction falls into the scope of the theory of canonical systems introduced by De Branges in \cite{DeBranges} (see e.g.~\cite{Remling2002, Remlingbook, Romanov} for reviews on the subject, and \cite{VVoperator} in the context of limiting random matrix models). Then, we rigorously define $\tSch_\tau$ as the image through the inverse transformation of this differential operator. Finally, we show that the eigenvalues / eigenvectors of $\tSch_\tau$ satisfy a system of SDEs naturally associated to the formal expression above.\\

This section is organized as follows. The first subsection \ref{subsec:canonicalsyst} recalls the basic material on the theory of canonical systems and its connection with Dirac equations. Subsection \ref{Subsec:Sch} then applies this material to $\Sch_\tau$ and presents the proofs of Theorem \ref{Th:Sch} and \ref{Th:Intensity}.

\subsection{Canonical systems}\label{subsec:canonicalsyst}

From now on, we write
$$ J := \begin{pmatrix} 0 & -1\\ 1&0\end{pmatrix}\;.$$
Note the identity $J^{-1} = - J$. Let us recall that we write $M^\intercal$ for the transpose of any matrix $M$.

\subsubsection{Canonical systems.} A canonical system is a system of first order differential equations of the form:
\begin{align}
&J \partial_t v(t) = z R(t) v(t), \quad t \in [0,1], \quad z \in \C \label{CanoSyst}\,,
\end{align}
where $R$ is an integrable function from $[0,1]$ into the space of $2 \times 2$ positive symmetric matrices.

\begin{remark}
To simplify the presentation, we have made several restrictive assumptions here. In the general theory, one works on an interval $[0,L)$ that can be unbounded and the matrix $R$ is only assumed to be locally integrable on $[0,L)$ and non-negative. Interestingly, Schr\"odinger operators, and in particular the operator $\cH_L$, can be transformed into a canonical system, but with a non-invertible $R$, see \cite{Remlingbook, Romanov}.
\end{remark}

For any boundary conditions $b_0, b_1 \in \R^2$, let us introduce the domain of the Hilbert space $L^2_R := L^2_R((0,1),\R^2)$:
\begin{align}
\mathcal{D}_{b_0,b_1}(\cT) := \Big\{f \in L^2_R\,: \;f \mbox{ A.C. on }(0,1), R^{-1}J f' \in L^2_R,\;f(0) \parallel b_0,\; f(1) \parallel b_1\Big\}\;. \label{domaincanonicalsyst}
\end{align}
where $\parallel$ means ``is parallel to''. It is proved in \cite{Remlingbook, Romanov} that $\mathcal{D}_{b_0,b_1}(\cT)$ is dense in $L^2_R$ and that 
\begin{align}
\cT := R^{-1}(t) J \partial_t \label{def:operatorcanosyst}
\end{align} 
is a well-defined, self-adjoint operator on $\cD_{b_0,b_1}(\cT)$. Note that $(\gl,\phi)$ is an eigenvalue/eigenvector of $\cT$ if and only if the solution $v$ of \eqref{CanoSyst} with $z=\gl$ satisfies the two b.c.~and $\phi$ is a multiple of $v$.\\

The resolvents of $\cT$ admit explicit kernels~\cite[Th 7.8]{Weidmann}. Fix a $z\in \C$ that does not lie in the spectrum of $\cT$. Let $v_z$ be the solution of \eqref{CanoSyst} that satisfies $v_z(0) = b_0$. Let also $\hat{v}_z$ be the solution of \eqref{CanoSyst} that satisfies $\hat{v}_z(1) \parallel b_1$ and $v_z(0)^\intercal J \hat{v}_z(0) = 1$. Note that the last quantity is nothing but the Wronskian of $v_z$ and $\hat{v}_z$. Then for any $f\in L^2_R$ we have
\begin{equation}\label{Eq:ResolventCano} (\cT-z)^{-1} f(t) := \int_0^1 \big(v_z(t) \hat{v}_z(s)^\intercal \;1_{t \le s} + \hat{v}_z(t) v_z(s)^\intercal\; 1_{s < t}\big) R(s) f(s)\;.\end{equation}

Our hypothesis imply that the functions $v_z$ and $\hat{v}_z$ are continuous on $[0,1]$, and therefore bounded. This readily implies that the operator $(\cT-z)^{-1}$ is Hilbert-Schmidt so that $\cT$ has discrete spectrum.

\subsubsection{Differential operators as canonical systems}
We would like to associate a self-adjoint operator to the following stochastic differential equations:
\begin{align}
J du_z(t) + dV(t) u_z(t) = z u_z(t) dt\;,\quad t \in [0,1]\;,\quad z \in \C\;.\label{eqdiffhgl}
\end{align}
Here $V$ is a $2\times2$ ``noise'' matrix: its entries are It\^o processes (they will be combinations of independent Brownian motions in the case of $\tSch_\tau$). Note that we have not set the initial condition yet.

\begin{remark}
Here we understand $dV(t) u_z(t)$ in the It\^o sense. It turns out that one can construct on a given probability space the solutions of the above SDE simultaneously for all $z \in \C$ and all possible initial conditions. The solutions are continuous w.r.t.~all parameters.
\end{remark}

We will see in this paragraph that we can transform this system into a canonical system \eqref{CanoSyst}. When $dV$ is function-valued, this is already known (see e.g. Example 1 in \cite{Romanov}). Here $dV$ has the regularity of white noise and we thus need to adapt the arguments. The basic idea remains however the same.\\

One introduces the evolving matrix $M := (u^{(N)} u^{(D)})$ where $u^{(N)}$, resp.~$u^{(D)}$, is the solution of \eqref{eqdiffhgl} with $z=0$ and starting from $u^{(N)} = (1,0)^\intercal$, resp.~$u^{(D)} = (0,1)^\intercal$. Of course the superscripts $N$ and $D$ refer to Neumann and Dirichlet. Note that
\begin{equation}\label{Eq:M} dM = J dV M\;,\quad M(0) = I\;,\end{equation}
and
\begin{align*}
d \det M = \big((JdV)_{11} + (JdV)_{22} + d<(JV)_{11}, (JV)_{22}> - d<(JV)_{12},(JV)_{21}> \big)\det M\,,
\end{align*}
with $\det M(0) = 1$. Therefore $M$ remains invertible at any time $t \in [0,1]$.\\

Coming back to the generic solution $u_z$ of the system \eqref{eqdiffhgl}, one considers the transformed process $v_z = M^{-1} u_z$. By computing $d(M v_z)$, one deduces that
\begin{align}\label{Eq:hz}
dv_z = - z M^{-1} J M v_z dt - M^{-1} d\langle M,v_z \rangle \,,
\end{align}
where for all $2\times 2$ matrix $A$ and vector $x \in \R^2$ whose entries are It\^o processes, we define their bracket through:
\begin{align*}
\langle A,x \rangle := \begin{pmatrix} \langle A_{11}, x_1 \rangle + \langle A_{12},x_2 \rangle \\ \langle A_{21}, x_1 \rangle + \langle A_{22},x_2 \rangle \end{pmatrix}\;.
\end{align*}
Equation \eqref{Eq:hz} shows that $v_z$ is differentiable and therefore $d\langle M,v_z \rangle$ vanishes. As a consequence
\begin{align}
J dv_z &= - z J M^{-1} J M v_z dt \notag \\
&= \frac{z}{\det M} M^\intercal \, M v_z dt\,. \label{cansystdirac}
\end{align}
where we used the identity $M^{-1} = -(\det M)^{-1} J M^\intercal \,J$ at the second line.

\begin{remark}
Note that $v_z$ is differentiable, while $u_z$ is Brownian-like. Our transformation removed the irregularity from the latter.
\end{remark}

Denote by $R := M^\intercal\, M/\det M$. Almost surely the matrix $R$ is a positive definite symmetric matrix at all times, and is integrable as its entries are continuous. From the results on canonical systems recalled before, we can associate a self-adjoint operator to the system \eqref{cansystdirac} by setting $\cT_{R} := R^{-1} J \partial_t$ on $L^2_{R}$, and by prescribing some boundary conditions $b_0$ and $b_1$. It acts on a domain $\mathcal{D}_{b_0,b_1}(\cT_R)$ explicited in \eqref{domaincanonicalsyst}.\\

We finally associate to the system of equations \eqref{eqdiffhgl} the following self-adjoint operator
\begin{align*}
\cS := M \cT_R M^{-1} = {\det M} \,(M^{-1})^\intercal \, J \partial_t \,M^{-1}\;,
\end{align*}
that acts on (recall that $M(0) = I$)
\begin{align*}
\mathcal{D}_{b_0, M(1) b_1}(\cS) := \big\{f \in L^2_{(\det M)^{-1} I}:\; M^{-1} f \in \mathcal{D}_{b_0,b_1}(\cT_R) \big\} \;.
\end{align*}
Since $\cS$ is a unitary\footnote{Note that $M$ is not a unitary matrix, but the transformation $f \in L^2_R \mapsto M f \in L^2_{(\det M)^{-1}I}$ is indeed unitary.} transformation of $\cT_R$, we deduce that it also has discrete spectrum. Furthermore, by conjugation we deduce the explicit expression of the kernel of its resolvents: for any $f\in L^2_{(\det M)^{-1} I}$ and any $z\in \C\backslash\R$
\begin{equation}\label{Eq:ResolventDirac}
(\cS-z)^{-1} f(t) :=  \int_0^1 \big(u_z(t) \hat{u}_z(s)^\intercal\, 1_{t \le s}(s) + \hat{u}_z(t) u_z(s)^\intercal \,1_{s < t}\big)\frac{1}{\det M(s)} f(s)\;,
\end{equation}
where $u_z$ and $\hat{u}_z$ are the solutions of \eqref{eqdiffhgl} that satisfy $u_z(0) =b_0$, $\hat{u}_z(1) \parallel b_1$ and $u_z(0)^\intercal J \hat{u}_z(0)=1$.\\

In view of Equation \eqref{eqdiffhgl}, the operator $\cS$ can be written formally
$$ \cS = \begin{pmatrix}
0 & -\partial_t\\
\partial_t & 0
\end{pmatrix} + dV\;,$$
with b.c.~$b_0$ at $0$ and $M(1)b_1$ at $1$. In general, the elements of $\cD(\cS)$ have Brownian like regularity but are not adapted (to the filtration of $V$): therefore one cannot apply It\^o's integration and the above expression for $\cS$ is only formal. On the other hand, the elements of $\cD(\cT_R)$ are absolutely continuous and the action of $\cT_R$ given in \eqref{def:operatorcanosyst} makes perfect sense. Let us mention that it would be possible to give a precise description of the action of $\cS$ on its domain using the theory of rough paths.\\
However, a rigorous connection with the formal equation of $\cS$ can be made at the level of the eigenvalues and eigenvectors:
\begin{lemma}\label{Lemma:EigenSDEs}
Almost surely for every $\gl\in \R$, the pair $(\gl,\varphi_\gl)$ is an eigenvalue / eigenvector of $\cS$ if and only if the solution $u_\gl$ of \eqref{eqdiffhgl} that starts from $u_\gl(0) = b_0$ is such that $u_\gl(1)$ is parallel to $M(1)b_1$ and $\varphi_\gl$ is a multiple of $u_\gl$.
\end{lemma}
\begin{proof}
First note that almost surely for every $\gl$, ($u_\gl$ is solution of \eqref{eqdiffhgl} starting from $b_0$ at time $0$ and is parallel to $M(1)b_1$ at time $1$) is equivalent to ($v_\gl$ is solution of \eqref{cansystdirac} starting from $b_0$ at time $0$ and is parallel to $b_1$ at time $1$).\\
Second, almost surely for every $\gl\in \R$, the pair $(\gl,\varphi_\gl)$ is an eigenvalue / eigenvector of $\cS$ if and only if $(\gl, M^{-1}\varphi_\gl)$ is an eigenvalue / eigenvector of $\cT_R$. Then the equation $\cT_R (M^{-1}\varphi_\gl) = \lambda M^{-1}\varphi_\gl$, together with the conditions that $M^{-1}\varphi_\gl$ is parallel to $b_0$ at $0$ and $b_1$ at $1$, is equivalent to saying that $M^{-1}\varphi_\gl$ is a multiple of $v_\gl$ where $v_\gl$ is the solution of \eqref{cansystdirac} that starts from $b_0$ at time $0$ and is parallel to $b_1$ at time $1$. We thus conclude.
\end{proof}

\subsection{Construction and properties of $\tSch_\tau$}\label{Subsec:Sch}

Consider
\begin{align*}
&dV :=  \frac{\sqrt\tau}{2} \mat{d\cB + \frac{1}{\sqrt{2}} d\cW_1}{\frac{1}{\sqrt{2}} d\cW_2}{\frac{1}{\sqrt{2}} d\cW_2}{d\cB - \frac{1}{\sqrt{2}} d\cW_1} \,,
\end{align*}
where $\cB$, $\cW_1$ and $\cW_2$ are independent Brownian motions. Choose $b_0$ and $b_1$ in a such a way that $b_0 = M(1)b_1 = (0,1)^\intercal$, where $M$ is defined in \eqref{Eq:M} (with $dV$ as above). A computation shows that $d\det M = 0$ so that $\det M \equiv 1$ and we thus set $R := M^\intercal \,M$. We apply the general construction of the previous subsection and set
$$ \tSch_\tau := 2 M \cT_R M^{-1}\;.$$

\begin{remark}\label{Rk:VV}
Let us comment here further on the link between the operator $\tSch_\tau$ and the ones appearing in the paper~\cite{VVoperator} of Valk\'o and Vir\'ag. All these operators are associated to canonical system of the form \eqref{def:operatorcanosyst}.
In their approach, Valk\'o and Vir\'ag decompose the matrix $R$ in the following way:
\begin{align*}
R = \frac{X^\intercal X}{\det X},\qquad X = \mat{1}{-x}{0}{y}\,.
\end{align*}
It enables them to encode those operators with an upper half-plane path given by $x +i y$. It turns out that the limiting eigenvalue point processes of various classical ensembles correspond to random paths $x+iy$ with a simple description (for example a hyperbolic Brownian motion with variance $4/\beta$ run in logarithmic time for the bulk limit of $\beta$-ensembles), which gives a nice geometric interpretation of those operators. Moreover, these canonical systems are unitarily equivalent to operators acting on $L^2_I$ via the conjugation by $X$.

Here, we take another type of decomposition of the matrix $R$, via 
\begin{align*}
R = \frac{M^\intercal M}{\det M},\quad dM = JdV M,\quad M(0) = I\,,
\end{align*}
with a \emph{full noise} matrix $M$. It turns out that this representation is appropriate to get convergence of the eigenvectors as well.
\end{remark}

The additional prefactor $2$ motivates the definition of $y_z$ (corresponding to $u_{z/2}$) that solves
\begin{align}\label{Eq:ymuSch}
dy_z(t) = -\frac{z}{2} Jy_z dt  + J dV(t) y_z(t)\;,\quad t \in [0,1]\;,\quad z \in \C\;,
\end{align}
with $y_z(0) = (0,1)^\intercal$. We thus have the following corollary of Lemma \ref{Lemma:EigenSDEs}.

\begin{corollary}\label{Cor:Eigenvalues}
Almost surely for every $\gl\in \R$, the pair $(\gl,\Psi_\gl)$ is an eigenvalue / eigenvector of $\tSch_\tau$ if and only if $y_\gl(1)$ is parallel to $(0,1)^\intercal$ and $\Psi_\gl$ is a multiple of $y_\gl$.
\end{corollary}

For any $z\in \C\backslash \R$, the resolvent writes:
\begin{equation}\label{Eq:ResolventDirac2}
(\tSch_\tau-z)^{-1} f(t) := \frac1{2} \int_0^1 \big(y_z(t) \hat{y}_z(s)^\intercal\, 1_{t \le s}(s) + \hat{y}_z(t) y_z(s)^\intercal \,1_{s < t}\big) f(s)\;,
\end{equation}
where $f \in L^2_I$, $y_z$ and $\hat{y}_z$ are solutions of \eqref{Eq:ymuSch} such that $y_z(0) = (0,1)^\intercal$, $\hat{y}_z(1) \parallel (0,1)^\intercal$ and $y_z(0)^\intercal J \hat{y}_z(0) = 1$. Note that $\hat{y}_z$ can be constructed by setting
\begin{equation}\label{Eq:hatyz} \hat{y}_z := v_z - \alpha y_z\;,\quad \alpha := \frac{\begin{pmatrix} 1 & 0\end{pmatrix} v_z(1)}{\begin{pmatrix} 1 & 0\end{pmatrix} y_z(1)}\;,\end{equation}
where $v_z$ is the solution of \eqref{Eq:ymuSch} that starts from $(1,0)^\intercal$ at time $0$.\\

Let us now introduce the polar coordinates, also called Pr\"ufer coordinates, associated to $y_\gl$ for any $\gl\in\R$ through the relation
$$ (y_\gl)_2 + i (y_\gl)_1 =: \Gamma_{\gl} e^{i \Theta_{\gl}}\;.$$
A simple computation shows that
\begin{equation}\label{Eq:Theta}\begin{split}
&d \Theta_{\gl}(t) = \frac{\gl}{2} dt + \frac{\sqrt \tau}{2} d\mathcal{B}(t) + \frac{\sqrt \tau}{2\sqrt 2} \Re(e^{2i \Theta_{\gl}(t)} d\cW(t))\,,\\
&d \ln \Gamma_{\gl}(t) = \frac{\tau}{8} dt  + \frac{\sqrt\tau}{2\sqrt 2} \Im(e^{2i \Theta_{\gl}(t)} d\cW(t))\,,
\end{split}
\end{equation}
where $\cW = (\cW_1 + i \cW_2)$ is a complex Brownian motion.

\begin{proof}[Proof of Theorem \ref{Th:Sch}]
It is a consequence of the material above, noticing that $\Theta_{\gl}(1) \in \pi\Z$ if and only if $y_\gl(1)$ is parallel to $(0, 1)^\intercal$.
\end{proof}

\begin{remark}
As we will see later on, almost surely for any $t\in [0,1]$ the phase $\gl \mapsto \Theta_{\gl}(t)$ is increasing. However $t \mapsto \lfloor \Theta_{\gl}(t)/\pi \rfloor$ is \emph{not} non-decreasing, while the phase associated to $1$-d Schr\"odinger operators satisfies this property, often called Sturm-Liouville property.
\end{remark}

\begin{remark}
We could have endowed the operator with other b.c. For instance, let $\tSch_{\tau,\ell}$ be defined similarly as $\tSch_\tau$ except that we impose $M(1)b_1 = (\sin \ell,\cos \ell)^\intercal$ for some $\ell \in [0,\pi)$. The eigenvalues of $\tSch_{\tau,\ell}$ are those $\lambda$ for which $\Theta_{\gl}(1)$ equals $\ell$ modulo $\pi$. We have the following scaling property of the family $\tSch_{\tau,\ell}$:
\begin{align*}
C_{-\ell} (\tSch_{\tau} + 2 \ell ) C^{-1}_{-\ell} \overset{(d)}{=} \tSch_{\tau,\ell}\,,
\end{align*}
where $C_{-\ell} := \mat{\cos(-\ell t)}{-\sin(-\ell t)}{\sin(-\ell t)}{\cos(-\ell t)}$.
Consequently the set of eigenvalues of $\tSch_{\tau,\ell}$ coincides in law with the point process $\Sch_\tau + 2\ell$. It is then easy to deduce that the point process $\Sch_\tau$ is invariant in law under translation by integer multiples of $2\pi$.
\end{remark}

We turn to the computation of the intensity measure of the point process of eigenvalues / eigenvectors of $\tSch_\tau$. We start with some preliminary results. First, we show that the number of points in $\Sch_\tau$ that fall in any given compact set has finite expectation.

\begin{lemma}\label{Lemma:Expect}
For any $\mu < \lambda$, we have $\E[\#\{ \Sch_\tau \cap [\mu,\lambda]\}] < \infty$.
\end{lemma}
\begin{proof}
Thanks to the characterization of the $\Sch_\tau$ point process \eqref{Characdiffusion:Schtau} and the monotonicity property of $\gl \mapsto \Theta_{\gl}(1)$, we have:
$$ \#\{ \Sch_\tau \cap [\mu,\lambda]\} \le \frac{\Theta_{\gl}(1) - \Theta_{\mu}(1)}{\pi} + 1\;,$$
hence it suffices to bound the expectation of $\Theta_{\gl}(1) - \Theta_{\mu}(1)$. For any $t\in [0,1]$, we have
$$ \Theta_{\gl}(t) - \Theta_{\mu}(t) = \frac{\lambda-\mu}{2}t + \frac{\sqrt \tau}{2\sqrt 2} \int_0^t \Re\big( (e^{2i \Theta_{\gl}(s)}-e^{2i \Theta_{\mu}(s)}) dW(s)\big)\;.$$
An elementary estimate on the bracket of the martingale, together with simple computations show that there exists a constant $C>0$ such that for all $t\in [0,1]$
$$ \E[|\Theta_{\gl}(t) - \Theta_{\mu}(t)|^2] \le C(1 + \int_0^t \E[|\Theta_{\gl}(s) - \Theta_{\mu}(s)|^2] ds)\;.$$
Gr\"onwall's Lemma then yields $\E[|\Theta_{\gl}(1) - \Theta_{\mu}(1)|^2] < \infty$ which suffices to conclude.
\end{proof}

Second we compute a change of measure, which is essentially the same as in~\cite[Proof of Lemma 3.6]{RV}.
\begin{lemma}\label{Lemma:Girsanov}
Fix $u\in [0,1]$ and let $\cB_1$ be a real Brownian motion on $[0,1]$ starting from $0$. Set $f^u(t) := (u - |u-t|)/2$ and $Y(t) := t \frac{\tau}{8} + \frac{\sqrt\tau}{2\sqrt 2} \cB_1(t)$ for $t\in [0,1]$. Then for any bounded measurable map $G$ on $\cC([0,1],\R)$ we have
$$\E\Big[G(Y) e^{(u-1) \frac{\tau}{4} + \frac{\sqrt\tau}{\sqrt 2} (\cB_1(u)-\cB_1(1))}\Big] = \E\Big[G\big(\frac{\tau}{4} f^u + \frac{\sqrt \tau}{2\sqrt 2} \cB_1\big) \Big]\;.$$
\end{lemma}
\begin{proof}
Set $d\bbQ := e^{(u-1) \frac{\tau}{4} + \frac{\sqrt\tau}{\sqrt 2} (\cB_1(u)-\cB_1(1))} d\P$ and observe that $\bbQ$ is a probability measure: we want to show that the law of $Y$ under $\bbQ$ coincides with the law of $\frac{\tau}{4} f^u + \frac{\sqrt \tau}{2\sqrt 2} \cB_1$ under $\P$.

The law of $Y$ is characterized by the laws of $(Y(t), \; t \in [0,u])$ and $(\tilde Y(t) := Y(t) - Y(u),\; t \in [u,1])$ and observe that under $\bbQ$, the process $(Y(t))_{0\le t \le u}$ is independent of $(\tilde{Y}(t))_{u\le t \le 1}$. 

Consequently, it suffices to compute separately the laws of these two processes. Since the exponential change of measure is independent of $(\cB(t))_{0\le t \le u}$, it is immediate that $(Y(t))_{0\le t \le u}$ has the same law under $\bbQ$ and $\P$. And this is coherent with the fact that
$$ \frac{\tau}{4} f^u(t) + \frac{\sqrt \tau}{2\sqrt 2} \cB_1(t) = Y(t)\;,\quad t\in [0,u]\;.$$
On the other hand, Girsanov's Theorem~\cite[Th VIII.1.7]{RevuzYor} shows that under $\bbQ$, the process $(\cB_1(t) - \cB_1(u) + \frac{\sqrt\tau}{\sqrt 2}(t-u))_{u\le t \le 1}$ is a Brownian motion starting from $0$ at time $u$. Consequently, under $\bbQ$ the process $(\tilde{Y}(t))_{u\le t \le 1}$ has the same law as the process 
$$\frac{\tau}{4} (f^u(t)-f^u(u)) + \frac{\sqrt \tau}{2\sqrt 2} (\cB_1(t) - \cB_1(u))\;,\quad u\le t \le 1\;,$$
under $\P$. This completes the proof.
\end{proof}

We now proceed with the computation of the intensity measure.

\begin{proof}[Proof of Theorem \ref{Th:Intensity}]
Assume that
\begin{equation}\label{Eq:Intensity2}
\E\Big[\sum_{\gl: \Theta_{\gl}(1) \in \pi\Z} G(\gl, \ln \Gamma_{\gl},\Theta_{\gl})\Big]
= \int \sum_{n \in \Z} \E\Big[G(\gl, \frac{\sqrt \tau}{2\sqrt{2}}\mathcal{B}_1 + \frac{\tau f^{U}}{4},\beta^{(n\pi)})\Big] g_{\lambda, \frac32 \tau}(2n\pi) d\gl\;.
\end{equation}
The identity
$$ \Psi_{\lambda} = \frac{\Gamma_{\gl}}{\|\Gamma_{\gl}\|_{L^2([0,1],\R)}} \begin{pmatrix} \sin \Theta_{\gl} \\ \cos \Theta_{\gl} \end{pmatrix}\;,$$
shows that $\Psi_{\lambda}$ is the image through a continuous map of $(\ln \Gamma_{\gl},\Theta_{\gl})$. From \eqref{Eq:Intensity2} we thus deduce that
\begin{equation*}
\E\Big[\sum_{\gl: \Theta_{\gl}(1) \in \pi\Z} G(\gl, \Psi_{\lambda})\Big]
= \int \sum_{n \in \Z} \E\Big[G(\gl, \frac{e^{Z}}{\big\| e^{Z} \big\|_{L^2([0,1])}} \begin{pmatrix} \sin \beta^{(n\pi)} \\\cos \beta^{(n\pi)} \end{pmatrix})\Big] g_{\lambda, \frac32 \tau}(2n\pi) d\gl\;,
\end{equation*}
where $Z := \frac{\sqrt \tau}{2\sqrt{2}}\mathcal{B}_1 + \frac{\tau f^{U}}{4}$. Theorem \ref{Th:Intensity} then follows from the equality in law
$$ \frac{e^{Z(t)}}{\big\| e^{Z} \big\|_{L^2([0,1])}} \stackrel{(d)}{=}  \frac{e^{\frac{\sqrt \tau}{2\sqrt{2}}\mathcal{B}_1(t-U) - \frac{\tau}8 |t-U| }}{\big\| e^{\frac{\sqrt \tau}{2\sqrt{2}}\mathcal{B}_1(\cdot-U) - \frac{\tau}8 |\cdot - U| } \big\|_{L^2([0,1])}}\;.$$

\bigskip

We are left with the proof of \eqref{Eq:Intensity2}. Recall that the process $(\Theta_{\gl}(t),\Gamma_{\gl}(t); t\in [0,1],\gl\in \R)$ is continuous in both variables and satisfies the SDEs \eqref{Eq:Theta}. Note that $ \Theta_{\gl}$ is differentiable (in fact real-analytic) with respect to $\gl$ (see e.g. \cite[Chap. V, Theorem 40]{Protter2005} or Theorem 24 of \cite{KVV}) and its derivative $\gl \mapsto \Theta_{\gl}(t)$ satisfies the following SDE
\begin{align*}
d(\partial_\gl \Theta_{\gl}) = \frac{1}{2}\, dt - \frac{\sqrt \tau}{\sqrt 2} (\partial_\gl \Theta_{\gl})\; \Im(e^{2 i\Theta_{\gl}} d\cW_t)\;.
\end{align*}
An application of It\^o's formula yields
\begin{equation}\label{Eq:partialTheta}
\partial_\gl \Theta_{\gl}(t) = \frac{1}{2} \int_0^t \exp\big(2 \ln \Gamma_{\gl}(u) - 2 \ln \Gamma_{\gl}(t)\big) du\;.
\end{equation}
We thus deduce that almost surely for all $t\in [0,1]$ and all $\gl\in\R$, we have $\partial_\gl \Theta_{\gl}(t) > 0$. This implies that almost surely $\lambda\mapsto \Theta_{\gl}(1)$ is a $\cC^1$-diffeomorphism. In the sequel, we denote by $\theta\mapsto \lambda(\theta)$ its inverse.\\

By standard approximation arguments, it suffices to take $G$ non-negative, bounded and continuous, and such that $G(\lambda,\cdot) = 0$ whenever $\lambda \notin [-a,a]$ for some $a>0$. Note that
$$ \sum_{\gl: \Theta_{\gl}(1) \in \pi\Z} G(\gl, \ln \Gamma_{\gl},\Theta_{\gl}) = \sum_{\theta\in \pi\Z} G(\lambda(\theta), \ln \Gamma_{\gl(\theta)},\Theta_{\gl(\theta)})\;.$$
By continuity, the right hand side is the almost sure limit as $\epsilon\downarrow 0$ of
\begin{align*}
X_\eps &:= \frac{1}{2\eps}\int_{\theta \in [-\eps,\eps] + \pi \Z} G(\lambda(\theta), \ln \Gamma_{\gl(\theta)}, \Theta_{\gl(\theta)})d\theta \\
&= \frac{1}{2\eps}\int_{\gl \in \R} 1_{\{\Theta_\gl \in [-\eps,\eps] + \pi \Z\}} G(\lambda,  \ln \Gamma_{\gl}, \Theta_{\gl}) \partial_\gl \Theta_{\gl}(1) d\gl\;.
\end{align*}
Recall that $G$ is compactly supported in its first variable. Provided that $\epsilon < \pi$, we see that almost surely $|X_\eps| \le \|G\|_\infty (\#\{\lambda \in [-a,a]: \Theta_{\gl}(1) \in \pi\Z\} + 2)$. The latter r.v.~has finite expectation by Lemma \ref{Lemma:Expect}. The Dominated Convergence and Fubini Theorems thus yield
\begin{equation}\label{Eq:GTheta}\begin{split}
&\E[\sum_{\gl: \Theta_{\gl}(1) \in \pi \Z} G(\gl,\ln \Gamma_{\gl}, \Theta_{\gl})]\\
&= \lim_{\eps \to 0}\frac{1}{2\eps}\int_{\gl \in  \R} \E[1_{\{\Theta_{\gl}(1) \in [-\eps,\eps] + \pi \Z\}} G(\lambda,  \ln \Gamma_{\gl}, \Theta_{\gl}) \partial \Theta_{\gl}(1)] d\gl
\end{split}\end{equation}

Fix $\lambda \in \R$ and recall the definition of the SDEs \eqref{Eq:Theta} and \eqref{Eq:partialTheta}. Observe that $\int_0^t \Re(e^{2i \Theta_{\gl}} d\cW)$ and $\int_0^t \Im(e^{2i \Theta_{\gl}} d\cW)$ are independent Brownian motions. Since $\cW$ is independent of $\cB$, we deduce that the process $\Theta_{\gl}$ is independent of $(\ln \Gamma_{\gl}, \partial_\lambda \Theta_{\lambda})$.\\
We now provide some identities in law on the process $\Theta_{\gl}$. First, the r.v.~$\Theta_{\gl}(1)$ has a Gaussian law, centered at $\lambda/2$ with variance $3\tau /8$. Moreover the process $\Theta_{\gl}$ conditioned to $\{\Theta_{\gl}(1) =x\}$ has the same law as $\beta^{(x)}$ where, for some Brownian motion $\cB_2$,
$$ d\beta^{(x)}(t) = x dt + \sqrt{\frac{3\tau}{8}} (d\cB_2(t) - \cB_2(1) dt)\;.$$
The process $\beta^{(x)}$ is a scaled Brownian bridge from $(0,0)$ to $(1,x)$ (note that the law of $\beta^{(x)}$ no longer depends on $\gl$).\\

Desintegrating the expectation appearing on the r.h.s.~of \eqref{Eq:GTheta} according to the law of $\Theta_{\gl}(1)$ we get
\begin{align*}
&\E[1_{\{\Theta_{\gl}(1) \in [-\eps,\eps] + \pi \Z\}} G(\lambda,  \ln \Gamma_{\gl}, \Theta_{\gl}) \partial \Theta_{\gl}(1)]\\
&= \int_{x\in\R} 1_{\{x\in [-\eps,\eps] + \pi\Z\}} \E\Big[G(\lambda,  \ln \Gamma_{\gl}, \beta^{(x)}) \partial \Theta_{\gl}(1)\Big] g_{\lambda/2, \frac38 \tau}(x)dx\;.
\end{align*}
Note that $g_{\lambda/2, \frac38 \tau}(x) = 2g_{\lambda, \frac32 \tau}(2x)$. Applying again the Dominated Convergence Theorem, we obtain
\begin{align*}
&\E[\sum_{\gl: \Theta_{\gl}(1) \in \pi \Z} G(\gl,\ln \Gamma_{\gl}, \Theta_{\gl})]\\
&= \int_{\gl \in  \R} \lim_{\eps \to 0}\frac{1}{2\eps} \int_{x\in\R} 1_{\{x\in [-\eps,\eps] + \pi\Z\}} 2\,\E\Big[G(\lambda,  \ln \Gamma_{\gl}, \beta^{(x)}) \partial \Theta_{\gl}(1)\Big] g_{\lambda, \frac32 \tau}(2x)dx d\gl\\
&=\int_{\gl \in  \R} \sum_{n\in\Z} 2\,\E\Big[G(\lambda,  \ln \Gamma_{\gl}, \beta^{(n\pi)})\partial \Theta_{\gl}(1)\Big] g_{\lambda, \frac32 \tau}(2n\pi) d\gl\;.
\end{align*}

It remains to compute the expectation that appears on the r.h.s. Recall that $\beta^{(n\pi)}$ is independent of the pair $(\ln \Gamma_{\gl}, \partial_\lambda \Theta_{\lambda})$. We already saw that $\cB_1(t) := \int_0^t \Im(e^{2i \Theta_{\gl}} d\cW)$ is a standard Brownian motion. Consequently
$$ \ln \Gamma_{\gl}(t) = t \frac{\tau}{8} + \frac{\sqrt \tau}{2\sqrt 2} \cB_1(t) =: Y(t)\;.$$
Using the explicit expression \eqref{Eq:partialTheta} of $\partial_\lambda \Theta_{\lambda}$ in terms of $\ln \Gamma_{\gl}$, we obtain
\begin{align*}
&2\, \E\Big[G(\lambda,  \ln \Gamma_{\gl}, \beta^{(n\pi)})\partial \Theta_{\gl}(1)\Big]= \int_0^1 \E\Big[G(\lambda,  Y , \beta^{(n\pi)}) e^{(u-1) \frac{\tau}{4} + \frac{\sqrt\tau}{\sqrt 2} (\cB_1(u)-\cB_1(1))}\Big] du\;.
\end{align*}
Lemma \ref{Lemma:Girsanov} allows to conclude.
\end{proof}

\section{Convergence of the operators}\label{sec:CVoperator}

We start this section with a detailed presentation of the unitary map that allows to construct $\tSchE$ from $\cH^{(E)}$ and with the proofs of claims on the eigenvectors made in the introduction. Then, we deal with the operator convergence of $\tSchE$ towards $\tSch_\tau$ and thus prove Theorem \ref{Th:CVoperator}. In the first two subsections, no assumption is made on the value of $E$, while in the subsequent subsections we always assume that $E \sim L/\tau$ for some fixed $\tau > 0$.

\subsection{The unitary transformation}

The operator $\cH_L$ is a generalized Sturm-Liouville operator. Its domain is made of (random) $H^1([0,L])$-functions that satisfy Dirichlet b.c.: namely, 
\begin{align*}
\cD(\cH_L) := \Big\{f \in L^2([0,L]):& \;f(0) = f(L) = 0,\; f \mbox{ A.C.}, \; f'- Bf  \mbox{ A.C.},\\
&\mbox{and } -(f'-Bf)' - Bf' \in L^2([0,L])\Big\}\,.
\end{align*}

Recall from the introduction that we consider the recentered operator $(L/\sqrt E)(\cH_L - E) + 2\ell_E$, and that we conjugate it with the rescaling map $g\mapsto g(L \,\cdot )$ that goes from $L^2([0,L],\R)$ into $L^2([0,1],\R)$. This yields the operator
\begin{equation}\label{Eq:HE}
{\cH}^{(E)} f := -\frac1{L\sqrt E} f'' + \sqrt{\frac{L}{E}} dB^{(L)} f + (2\ell_E - L\sqrt E) f\;,\quad x\in [0,1]\;,
\end{equation}
whose domain is the image through the rescaling map of the domain of $\cH_L$. In particular any element $f$ of this domain belongs to $H^1([0,1])$ and satisfies Dirichlet b.c.: $(f(0),f'(0))$ and $(f(1),f'(1))$ are parallel to $(0,1)$.\\

We now recall the two steps presented in the introduction on the construction of $\tSchE$.\\

\emph{Step 1: From $\R$ to $\R^2$.} 
Recall from the introduction the matrix $T$ and the map $\iota: L^2((0,1),\R) \to L^2_T((0,1),\R^2)$ that sends any $f\in L^2((0,1),\R)$ on its (canonical) equivalent class in $L^2_T((0,1),\R^2)$. Since $\iota$ is unitary, the operator
$$ \HH^{(E)} := \iota {\cH}^{(E)} \iota^{-1}\;.$$
is a self-adjoint operator on $L^2_T((0,1),\R^2)$ with domain $\cD(\HH^{(E)}) := \iota (\cD({\cH}^{(E)}))$. It is equivalent to ${\cH}^{(E)}$: the eigenvalues and normalized eigenvectors of $\HH^{(E)}$ are given by $(\lambda, \iota {\varphi}^{(E)}_\gl)$ where $(\lambda, {\varphi}^{(E)}_\gl)$ are the eigenvalues and normalized eigenvectors of $\cH^{(E)}$.\\
Any element $f$ of $L^2_T$ is an equivalent class of $L^2_I$: its second coordinate $f_2$ is ``arbitrary''. However, in our context there is a convenient representative given by $Pf := (f_1,\frac1{L\sqrt E} f_1')$. Note that $Pf \in L^2_I$ provided $f_1 \in H^1((0,1),\R)$; this holds in particular for any $f\in \cD(\HH^{(E)})$. Then, any $f\in \cD(\HH^{(E)})$ satisfies Dirichlet b.c.~in the following sense: $Pf(0)$ and $Pf(1)$ are parallel to $(0,1)^\intercal$.

Our next lemma establishes the relationship between $\HH^{(E)}$ and the collection of SDEs \eqref{SDEs:u}.
\begin{lemma}\label{Lemma:RR2}
Almost surely the eigenvalues of $\HH^{(E)}$ are those $\gl$ for which the solution $(u_\gl^{(E)}, (u_\gl^{(E)})')$ of \eqref{SDEs:u} is parallel to $(0,1)^\intercal$ at time $1$. In that case the associated eigenvector $\bs{\varphi}^{(E)}_\gl$ is a multiple of the equivalence class associated to $u_\gl^{(E)}$ and we have $P \bs{\varphi}^{(E)}_\gl = (\varphi^{(E)}_\gl, \frac1{L\sqrt E} (\varphi^{(E)}_\gl)')^\intercal$.
\end{lemma}
\begin{proof}
$\gl$ is an eigenvalue of $\HH_L^{(E)}$ if and only if $\gl$ is an eigenvalue of $\cH^{(E)}$ if and only if the solution $u^{(E)}_\gl$ of \eqref{Eq:uE} vanishes at time $1$. The latter is equivalent with: the solution $(u_\gl^{(E)}, (u_\gl^{(E)})')$ of \eqref{SDEs:u} is parallel to $(0,1)^\intercal$ at time $1$. The theory of Sturm-Liouville operators shows that the eigenvector $\varphi^{(E)}_\gl$ of $\cH^{(E)}$ is then a multiple of $u^{(E)}_\gl$. Since $\bs{\varphi}^{(E)}_\gl = \iota \varphi^{(E)}_\gl$ we deduce that it is a multiple of the equivalence class of $(u_\gl^{(E)}, (u_\gl^{(E)})')$. This implies that $P \bs{\varphi}^{(E)}_\gl = (\varphi^{(E)}_\gl, \frac1{L\sqrt E} (\varphi^{(E)}_\gl)')^\intercal$.
\end{proof}

\emph{Step 2: Unrotate.} Recall $C_{E'}$ defined in \eqref{def:CE'} and note that $C_{E'}^{-1} = C_{E'}^\intercal$. We set $R_{E'} := C_{E'} T C_{E'}^\intercal$. The evolving rotation matrix $C_{E'}$ can be viewed as a unitary map from $L^2_T$ into $L^2_{R_{E'}}$. Indeed, for any $f\in L^2_T$ and for any representative $g\in L^2_I$ of $f$ we define $C_{E'} f$ as the equivalent class of $C_{E'} g$. It is simple to check that this definition is independent of the choice of the representative.\\
We then introduce the operator
$$ \tSchE := C_{E'} {\HH}^{(E)} C_{E'}^{-1}\;,$$
acting on the domain $\cD(\tSchE) := C_{E'} \cD({\HH}^{(E)}) \subset L^2_{R_{E'}}([0,1],\R^2)$. Recall that $E'$ is taken equal to $L\sqrt{E} - \ell_E$: the order $1$ correction $\ell_E = \{L\sqrt E\}_\pi$ has been chosen in such a way that any element $f \in \cD(\tSchE)$ satisfies Dirichlet b.c., that is, $P_{E'} f(0)$ and $P_{E'} f(1)$ are parallel to $(0,1)^\intercal$ where $P_{E'} := C_{E'} P C_{E'}^{-1}$.\\
Note that the eigenvalues and normalized eigenvectors of $\tSchE$ are given by $(\lambda, C_{E'}\iota {\varphi}^{(E)}_\gl)$. Our next result connects $\tSchE$ with the collection of processes $(y^{(E)}_z,z\in \C)$ (this is the same equation as in \eqref{Eq:yE}, except it is written here for a complex $z$)
\begin{equation}\label{Eq:yEz}\begin{split}
&\bigg(\begin{pmatrix} 0 & -\partial_t\\ \partial_t & 0\end{pmatrix} + \frac12 {\sqrt{\frac{L}{E}}} \begin{pmatrix} dB^{(L)} + \frac1{\sqrt 2} dW_1^{(L)} & \frac1{\sqrt 2} dW_2^{(L)}\\ \frac1{\sqrt 2} dW_2^{(L)} & dB^{(L)} - \frac1{\sqrt 2} dW_1^{(L)}\end{pmatrix} + \ell_E (2R_{E'}-I)\bigg) y_z^{(E)}\\
&= z R_{E'}y_z^{(E)} \;,\end{split}\end{equation}
with initial condition $y_z^{(E)}(0) = (0,1)^\intercal$.

\begin{lemma}\label{Lemma:Unitary}
The eigenvalues of $\tSchE$ are those $\gl \in\R$ for which the solution $y_\gl^{(E)}$ of \eqref{Eq:yEz} is parallel to $(0,1)^\intercal$ at time $1$. The corresponding normalized eigenvector $\psi^{(E)}_\gl$ is then a multiple of the equivalence class in $L^2_{R_{E'}}$ of $y_\gl^{(E)}$, and we have (recall \eqref{Eq:PsiE})
\begin{equation*}P_{E'} \psi^{(E)}_\gl = C_{E'} \begin{pmatrix} {\varphi}^{(E)}_\gl \\ \frac1{L\sqrt E} ({\varphi}^{(E)}_\gl)' \end{pmatrix} = \Psi_\gl^{(E)}\;.\end{equation*}
\end{lemma}
\begin{proof}
By definition of $\tSchE$, $\gl$ is an eigenvalue of $\tSchE$ if and only if $\gl$ is an eigenvalue of $\HH^{(E)}$. By the previous lemma, this is equivalent with: $(u_\gl^{(E)},(u_\gl^{(E)})')$ of \eqref{SDEs:u} is parallel to $(0,1)$ at time $1$. But since $y_\gl^{(E)} = C_{E'} ( u_\gl^{(E)}, \frac1{L\sqrt E} (u_\gl^{(E)})')^\intercal$, we deduce that it is in turn equivalent to: $y_\gl^{(E)}$ of \eqref{Eq:yE} is parallel to $(0,1)^\intercal$ at time $1$. In that case, since $\psi^{(E)}_\gl = C_{E'} \bs{\varphi}^{(E)}_\gl$ and given the previous lemma, we deduce that $\psi^{(E)}_\gl$ is a multiple of the equivalence class of $y_\gl^{(E)}$. The last identity then easily follows.
\end{proof}

\begin{remark}
One can show that for any $f$ in the domain of $\tSchE$, the function $g:= \tSchE f$ satisfies the equation:
\begin{equation}\label{Eq:SchE}\begin{split} &\bigg(2\begin{pmatrix} 0 & -\partial_t\\ \partial_t & 0\end{pmatrix} + {\sqrt{\frac{L}{E}}} \begin{pmatrix} dB^{(L)} + \frac1{\sqrt 2} dW_1^{(L)} & \frac1{\sqrt 2} dW_2^{(L)}\\ \frac1{\sqrt 2} dW_2^{(L)} & dB^{(L)} - \frac1{\sqrt 2} dW_1^{(L)}\end{pmatrix} + 2\ell_E (2R_{E'}-I)\bigg) P_{E'}f\\
&= 2R_{E'} g\;.
\end{split}
\end{equation}
Note the similarity with \eqref{DefopSch}.
\end{remark}

\begin{remark}
In the introduction, we mentioned that the space generated by the family $\{\Psi_\gl^{(E)}\}_\gl$ is not dense in $L^2_I$. Indeed, if we set $f:= C_{E'} (0,1)^\intercal$, we observe that
$$ \langle f , \Psi_\gl^{(E)}\rangle_{L^2_I} = \langle (0,1)^\intercal , ( \varphi_\gl^{(E)} , \frac1{L\sqrt E} (\varphi_\gl^{(E)})' )^\intercal \rangle_{L^2_I} = \int \frac1{L\sqrt E} (\varphi_\gl^{(E)})' = 0\;.$$
\end{remark}

\subsection{The resolvents and the precise statement of Theorem \ref{Th:CVoperator}}\label{Subsec:PreciseResolv}

Our goal is now to prove convergence of the resolvents of $\tSchE$ to those of $\tSch_\tau$. While the resolvents of the latter are operators on $L^2_I$, the resolvents of $\tSchE$ are defined on the quotient space $L^2_{R_{E'}}$: our first task is to extend these resolvents into well-defined, bounded operators on $L^2_I$.

\medskip

Of course, such an extension is \emph{far} from being unique. We will see that the extension that we opt for is related to the functions $\{\Psi^{(E)}_\gl\}_\gl$: in light of the statement of Theorem \ref{Th:Joint2}, this justifies a posteriori our choice.

\medskip

Fix $z \in \C \backslash \R$. For any $g \in L^2_{I}$, let $\dot{g}$ be its equivalent class in $L^2_{R_{E'}}$ and set
$$ \overline{(\tSchE - z)^{-1}} g := P_{E'} (\tSchE - z)^{-1} \dot{g}\;.$$
\begin{figure}[h!]
     \includegraphics[width = 6cm]{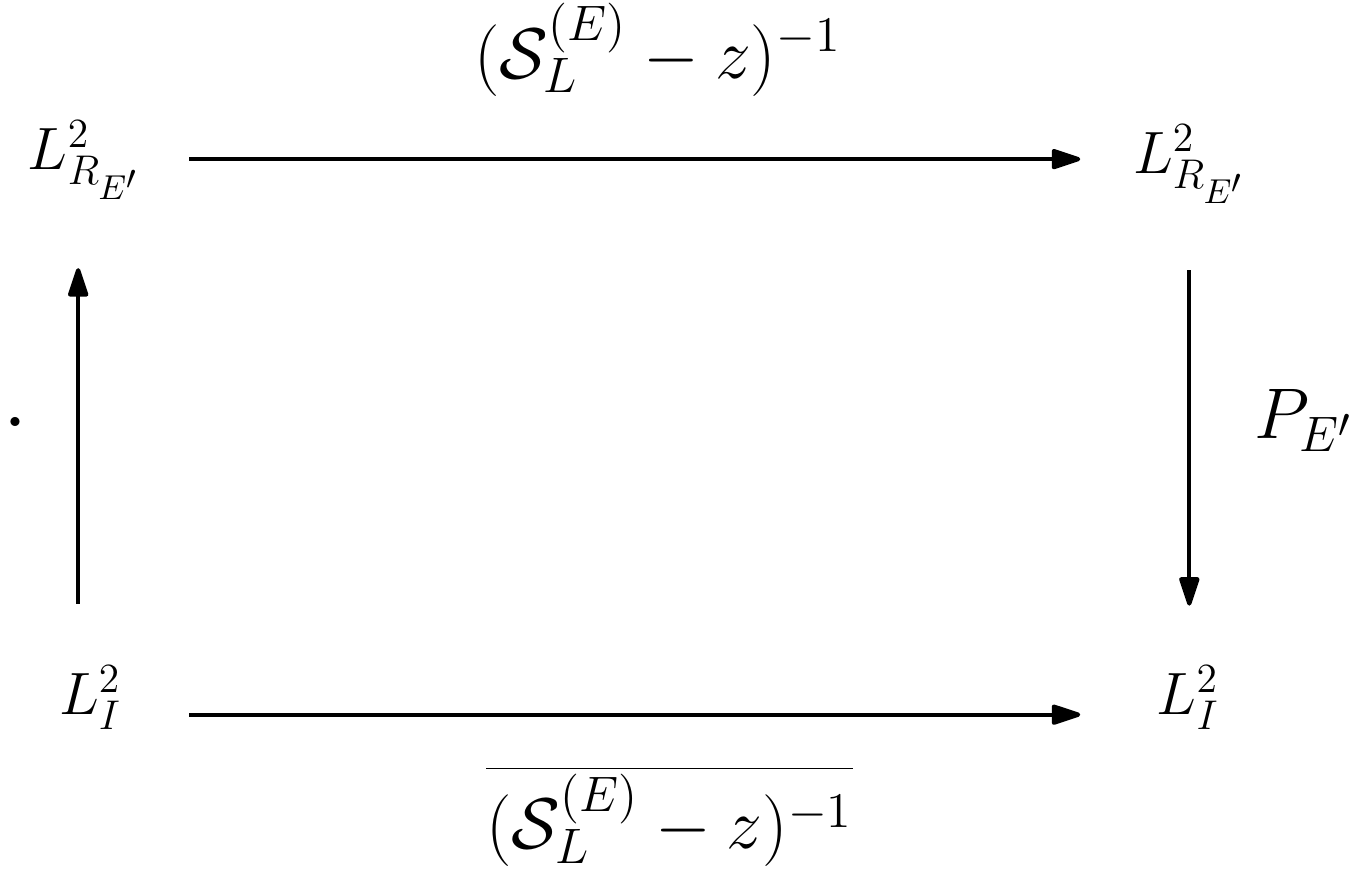}
     \caption{Extension of $(\tSchE-z)^{-1}$.}\label{Fig:Extension}
\end{figure}
\begin{remark}
The overline in the notation $\overline{(\tSchE - z)^{-1}}$ should not be confused with the complex conjugate. We use it to indicate that $\overline{(\tSchE - z)^{-1}}$ is an ``extension'' of the initial operator $(\tSchE - z)^{-1}$ to the space $L^2_I$.
\end{remark}
This definition is illustrated on Figure \ref{Fig:Extension}. We have composed $(\tSchE - z)^{-1}$ to the right with the (canonical) projection from $L^2_{I}$ to $L^2_{R_{E'}}$: this is clearly the only reasonable operation to apply here. On the other hand, we have composed it to the left with the densely defined injection $P_{E'}$ from $L^2_{R_{E'}}$ into $L^2_{I}$ and this may seem arbitrary. However note that for any eigenvalue $\gl$ of $\tSchE$ we find
$$ \overline{(\tSchE - z)^{-1}} \Psi^{(E)}_\gl = (\lambda-z)^{-1} \Psi^{(E)}_\gl\;,$$
thus justifying a posteriori the composition with $P_{E'}$.

Our next proposition provides explicitly the kernel of $\overline{(\tSchE - z)^{-1}}$. Let $v^{(E)}_z$ be the solution of \eqref{Eq:yEz} starting from $v_z^{(E)}(0) = (1,0)^\intercal$. Note that $y^{(E)}_z$ and $v^{(E)}_z$ are respectively the Dirichlet and Neumann solutions of \eqref{Eq:yEz}, and that these functions live in $L^2([0,1], \C^2)$. We then set
\begin{equation}\label{Eq:hatyzE} \hat{y}^{(E)}_z = v^{(E)}_z - \alpha^{(E)}\, y^{(E)}_z  \;,\quad \alpha^{(E)} := \frac{\begin{pmatrix} 1 & 0\end{pmatrix} v^{(E)}_z(1)}{\begin{pmatrix} 1 & 0\end{pmatrix} y^{(E)}_z(1)}\;.\end{equation}
Note that $\hat{y}^{(E)}_z$ is a solution of \eqref{Eq:yEz} which is parallel to $(0,1)^\intercal$ at time $1$ (but not at time $0$), and whose Wronskian with $y^{(E)}_z$ equals $1$.
\begin{proposition}
For any $g\in L^2_I$, we have
\begin{equation}\label{Eq:ResolvE}\begin{split}
\overline{(\tSchE - z)^{-1}} g(t) = \int \Big( \hat{y}_z^{(E)}(t)  y^{(E)}_z(s)^\intercal \un_{s\le t} +  {y}^{(E)}_z(t)  \hat{y}_z^{(E)}(s)^\intercal \un_{s > t} \Big) R_{E'}(s) \;g(s) ds\;.
\end{split}\end{equation}
As a consequence, there is a constant $C>0$ independent of $E$ s.t.~the operator norm satisfies
$$ \| \overline{(\tSchE - z)^{-1}}  \| \le C  \| y^{(E)}_z \|_\infty \| \hat{y}_z^{(E)}\|_\infty\;.$$
\end{proposition}
\begin{proof}
Let $u^{(E)}_z$ be the solution of \eqref{Eq:uE} starting from $(u^{(E)}_z(0),(u^{(E)}_z)'(0))=(0,L\sqrt E)$. Let $\hat{u}^{(E)}_z$ be the solution of \eqref{Eq:uE} that is parallel to $(0,1)^\intercal$ at time $1$ and whose Wronskian with $u^{(E)}_z$ satisfies
$$(u^{(E)}_z)'(0) \hat{u}^{(E)}_z(0) - u^{(E)}_z(0) (\hat{u}^{(E)}_z)'(0) = L\sqrt E\;.$$
By the theory of Sturm-Liouville operators, we have an explicit integral form for the resolvent of the rescaled operator ${\cH}^{(E)}$ (defined in \eqref{defopHE}) given by
$$ ({\cH}^{(E)}-z)^{-1} f(t) = \int_0^1 \big(\hat{u}^{(E)}_z(t) u^{(E)}_z(s) \un_{s \le t} + u^{(E)}_z(t) \hat{u}^{(E)}_z(s) \un_{t < s}(s)\big) f(s) ds\;,$$
for $f \in L^2([0,1],\R)$.
Note the identities
$$ y^{(E)}_z = C_{E'} \begin{pmatrix} u^{(E)}_z \\ \frac1{L\sqrt E} (u^{(E)}_z)' \end{pmatrix}\;,\quad \hat{y}^{(E)}_z = C_{E'}\begin{pmatrix} \hat{u}^{(E)}_z\\ \frac1{L\sqrt E} (\hat{u}^{(E)}_z)'\end{pmatrix}\;.$$
Since $C_{E'}^\intercal J C_{E'} = J$, it is easy to check that
$$ \frac{1}{L\sqrt{E}}((u^{(E)}_z)'(0) \hat{u}^{(E)}_z(0) - u^{(E)}_z(0) (\hat{u}^{(E)}_z)'(0)) = y_z^{(E)}(0)^\intercal J \hat{y}^{(E)}_z(0) = 1\;.$$
Furthermore from the identities $P_{E'} = C_{E'} P C_{E'}^{-1}$ and $\tSchE = C_{E'} \iota {\cH}^{(E)} \iota^{-1}C_{E'}^{-1}$, we obtain
$$ P_{E'} (\tSchE - z)^{-1} = C_{E'} P \iota ({\cH}^{(E)} -z)^{-1} \iota^{-1}C_{E'}^{-1}\;.$$
Now observe that we have
$$ u^{(E)}_z \iota^{-1}C_{E'}^{-1} \dot{g} = \begin{pmatrix} u^{(E)}_z & \frac1{L\sqrt E} (u^{(E)}_z)' \end{pmatrix} T C_{E'}^{-1} g = ({y}^{(E)}_z)^\intercal R_{E'} g\;,$$
and similarly
$$ \hat{u}^{(E)}_z \iota^{-1}C_{E'}^{-1} \dot{g} = \begin{pmatrix} \hat{u}^{(E)}_z & \frac1{L\sqrt E} (\hat{u}^{(E)}_z)' \end{pmatrix} T C_{E'}^{-1} g = (\hat{y}^{(E)}_z)^\intercal R_{E'} g\;.$$
Furthermore
$$ C_{E'} P \iota \hat{u}^{(E)}_z = \hat{y}^{(E)}_z\;,\quad C_{E'} P \iota {u}^{(E)}_z = {y}^{(E)}_z\;.$$
Putting everything together, we deduce the asserted expression for $P_{E'} (\tSchE - z)^{-1} \dot{g}$. Finally, since the entries of the matrix $R_{E'}$ are all bounded by $1$, the bound on the operator norm follows.
\end{proof}

Let $A_n, A$ be random bounded operators on $L^2_I$. Recall that $A_n \to A$ in law for the strong operator topology if the finite-dimensional marginals of the process $(A_n f, f\in L^2_I)$ converge in law to those of $(A f, f\in L^2_I)$. Furthermore $A_n\to A$ in law for the norm operator topology if the process $(A_n f, f\in L^2_I)$ converges in law for the topology of uniform convergence on bounded sets to $(A f, f\in L^2_I)$.
The precise statement of Theorem \ref{Th:CVoperator} is then:
\begin{theorem*}[Strong-resolvent convergence: precise statement]
Fix $\tau > 0$ and consider $E=E(L) \sim L/\tau$. As $L\to\infty$ and for any given $z\in\C\backslash\R$, the operator $\overline{(\tSchE-z)^{-1}}$ converges in law towards $(\tSch_\tau-z)^{-1}$ for the strong operator topology. However it does not converge in law to $(\tSch_\tau-z)^{-1}$ for the norm operator topology.
\end{theorem*}


\subsection{Strong resolvent convergence}\label{subsec:StrongResolventCV}

Recall from \eqref{Eq:ResolvE} and \eqref{Eq:ResolventDirac2} the expressions of the resolvents at stake. These resolvents depend respectively on the pair of processes $(y^{(E)}_z,\hat{y}^{(E)}_z)$, and $(y_z,\hat{y}_z)$. The main technical step consists in showing convergence in law of the former towards the latter.\\

Recall from \eqref{Eq:hatyzE} and \eqref{Eq:hatyz} that $\hat{y}^{(E)}_z = v^{(E)}_z - \alpha^{(E)}\, y^{(E)}_z$ and $\hat{y}_z = v_z - \alpha\, y_z$. Note that almost surely $\alpha^{(E)}$ is neither $0$ nor $\infty$. Indeed, suppose for instance that with positive probability $\alpha^{(E)}=\infty$, then it means that $y_z^{(E)}$ satisfies Dirichlet b.c.~at $0$ and $1$ so that $z$ is a non-real eigenvalue of the self-adjoint operator $\tSchE$, thus yielding a contradiction. The reasoning is the same for $\alpha^{(E)} = 0$ (with a contradiction with Neumann b.c.). Similarly almost surely $\alpha$ is neither $0$ nor $\infty$.\\

Our main technical step is the following result, whose proof is postponed to the next section.

\begin{proposition}\label{Prop:CVyv}
For any $z \in \C\backslash \R$, the process $(y^{(E)}_{z},v^{(E)}_{z})$ converges in law to $(y_{z},v_{z})$ for the topology of uniform convergence.
\end{proposition}

With this proposition at hand, we can proceed with the proof of the first part of the theorem.

\begin{proof}[Proof of Theorem \ref{Th:CVoperator} - Strong resolvent convergence]
To prove strong resolvent convergence, it suffices to show that for any fixed $g_1,\ldots,g_n \in L^2_I$, the vector $(\overline{(\tSchE - z)^{-1}} g_i)_{1\le i \le n}$ converges in law to $((\tSch_\tau - z)^{-1} g_i)_{1\le i \le n}$. Note that the resolvents at stake are measurable functions of the processes $(y^{(E)}_{z},v^{(E)}_{z})$ and $(y_{z},v_{z})$. We thus combine Proposition \ref{Prop:CVyv} and Skorohod's Representation Theorem, and work under a coupling for which $(y^{(E)}_{z},v^{(E)}_{z})$ converges almost surely to $(y_{z},v_{z})$. It now suffices to prove that for any function $g \in L^2_I$, $\overline{(\tSchE - z)^{-1}} g$ converges in probability to $(\tSch_\tau - z)^{-1} g$.\\

Recall that the norm of the operator $\overline{(\tSchE - z)^{-1}}$ is bounded by a constant times $\|y^{(E)}_{z}\|_\infty \|\hat{y}^{(E)}_{z}\|_\infty$, and similarly for the norm of $(\tSch_\tau-z)^{-1}$. From the almost sure uniform convergence of $(y^{(E)}_{z},v^{(E)}_{z})$ towards $(y_{z},v_{z})$, we deduce that almost surely the norms of $(\overline{(\tSchE - z)^{-1}})_{L>1}$ and $(\tSch_\tau-z)^{-1}$ are uniformly bounded. Since smooth functions are dense in $L^2_I$, we can restrict ourselves to considering smooth functions $g:[0,1]\to \R^2$ in the sequel.\\

From Proposition \ref{Prop:CVyv}, we deduce that the coefficient $\alpha^{(E)}$ converges almost surely to $\alpha$, and that the pair $(y^{(E)}_z,\hat{y}^{(E)}_z)$ converges almost surely to the pair $(y_z,\hat{y}_z)$.\\

Let us rewrite the resolvents in the following way
$$ \overline{(\tSchE - z)^{-1}} g(t) = \hat{y}^{(E)}_z(t) u^{(E)}(t) + y^{(E)}_z(t) \hat{u}^{(E)}(t)\;,\quad t\in [0,1]\;,$$
with
$$ u^{(E)}(t) := \int_0^t y_z^{(E)}(s)^\intercal R_{E'}(s) g(s) ds\;,\quad \hat{u}^{(E)}(t) := \int_t^1 \hat{y}_z^{(E)}(s)^\intercal R_{E'}(s) g(s) ds\;.$$
Similarly
$$ (\tSch_\tau - z)^{-1} g(t) = \hat{y}_z(t) u(t) + y_z(t) \hat{u}(t)\;,\quad t\in [0,1]\;,$$
with
$$ u(t) := \frac12 \int_0^t y_z(s)^\intercal g(s) ds\;,\quad \hat{u}(t) := \frac12 \int_t^1 \hat{y}_z(s)^\intercal g(s) ds\;.$$
Then we have
\begin{align*}
\| \overline{(\tSchE - z)^{-1}} g - (\tSch_\tau - z)^{-1} g\|_{L^2_I} &\le  \| \hat{y}^{(E)}_z u^{(E)} - \hat{y}_z u \|_{L^2_I}+ \|y^{(E)}_z \hat{u}^{(E)} - y_z \hat{u} \|_{L^2_I}\;.
\end{align*}
The arguments to bound the two terms on the r.h.s.~are the same, so we provide the details only for the first. We write
\begin{align*}
\| \hat{y}^{(E)}_z u^{(E)} - \hat{y}_z u \|_{L^2_I} \le \| (\hat{y}^{(E)}_z - \hat{y}_z) u^{(E)} \|_{L^2_I} + \| \hat{y}_z (u^{(E)} - u)\|_{L^2_I}
\end{align*}
The a.s.~convergence of $\hat{y}^{(E)}_z$ to $\hat{y}_z$ ensures that the first term on the r.h.s.~goes to $0$ almost surely. Regarding the second term, we have
$$  \| \hat{y}_z (u^{(E)} - u)\|_{L^2_I} \le \sup_{t\in [0,1]} |\hat{y}_z| \sup_{t\in [0,1]} |u^{(E)}(t) - u(t)|\;.$$
and it remains to show that $\sup_{t\in [0,1]} |u^{(E)}(t) - u(t)|$ goes to $0$ in probability. Observe that
$$ u^{(E)}(t) - u(t) =  \int_0^t y_z^{(E)}(s)^\intercal (R_{E'}(s) - (1/2)I) g(s) ds + \frac12 \int_0^t (y_z(s) - y_z^{(E)})^\intercal g(s) ds\;.$$
From the almost sure convergence of ${y}^{(E)}_z$ to ${y}_z$, we deduce that the second term goes to $0$ almost surely. Note that
$$ R_{E'}(s) - (1/2)I = \frac12 \begin{pmatrix} \cos 2E's & \sin 2E's\\ \sin 2E's & -\cos 2E's\end{pmatrix}\;,$$
so that the first term is a linear combination of expressions of the form
$$ \int_0^t f(s) (y^{(E)}_z)_i(s) g_j(s) ds\;,$$
where $f(s)$ is either $\cos 2E's$ or $\sin 2E's$ and $i,j \in \{1,2\}$. Since $E'\to\infty$ as $L\to\infty$, the Riemann-Lebesgue Lemma should imply that this term goes to $0$ almost surely as $L\to\infty$: however, we are not exactly within the scope of this lemma since the (random) function $y^{(E)}_z$ \emph{depends} on $L$. We thus rely on Lemma \ref{Lemma:RL}, which is stated below, and this suffices to conclude.
\end{proof}

\subsection{Absence of norm resolvent convergence}

Set
$$ g_E(t) := (\sin^2 E't , -\sin E't \cos E't)^\intercal\;,\quad t\in [0,1]\;.$$
It is easy to check that
$$ \| g_E \|_{L^2_I} \to \frac12\;,\quad L\to\infty\;,$$
and therefore $(g_E)_{L\ge 1}$ remains in a bounded set of $L^2_I$.\\
Note that $R_{E'} g_E = 0$ so that $\overline{(\tSchE - z)^{-1}} g_E = 0$. To conclude, it suffices to show that as $L\to\infty$, with positive probability $(\tSch_\tau - z)^{-1} g_E$ does not converge to $0$ in $L^2_I$.\\

Recall that
$$ (\tSch_\tau - z)^{-1} g_E(t) = \hat{y}_z(t) u_E(t) + y_z(t) \hat{u}_E(t)\;,\quad t\in [0,1]\;,$$
with
$$ u_E(t) := \frac12 \int_0^t y_z(s)^\intercal g_E(s) ds\;,\quad \hat{u}_E(t) := \frac12 \int_t^1 \hat{y}_z(s)^\intercal g_E(s) ds\;.$$
By the Riemann-Lebesgue Lemma, almost surely $u_E, \hat{u}_E$ converge pointwise to $u,\hat{u}$ where
$$ u(t) = \frac14 \int_0^t (y_z)_1(s)ds\;,\quad \hat{u}(t) = \frac14 \int_t^1 (\hat{y}_z)_1(s)ds\;.$$
Note that $|u_E|_\infty$ and $|\hat{u}_E|_\infty$ are almost surely bounded by $(|y_z|_\infty + |\hat{y}_z|_\infty)$. Therefore by the Dominated Convergence Theorem, almost surely $(\tSch_\tau - z)^{-1} g_E$ converges in $L^2_I$ to
$$t\mapsto \hat{y}_z(t) u(t) + y_z(t) \hat{u}(t)\;.$$
Since $y_z(t)$ and $\hat{y}_z(t)$ are linearly independent for all $t\in [0,1]$, we deduce that the $L^2_I$-norm of the latter vanishes if and only if $u(t) = \hat{u}(t) = 0$ for almost every $t\in [0,1]$. The latter property would imply that $(y_z)_1$ and $(\hat{y}_z)_1$ are identically $0$, which is not true almost surely. Consequently, almost surely $(\tSch_\tau - z)^{-1} g_E$ converges in $L^2_I$ to a non-degenerate limit, thus concluding the proof of Theorem \ref{Th:CVoperator}.

\section{Convergence of the SDEs}\label{Sec:SDEs}

In this section, we prove Theorems \ref{Th:Joint1} and \ref{Th:Joint2}. The arguments are relatively elementary: we show convergence of the system of SDEs associated with the operator $\tSchE$ towards its counterpart for $\tSch_\tau$. At the end of the section, we present the proof of Proposition \ref{Prop:CVyv}, since the arguments are small modifications of the previous ones. Until the end of the section we always assume that $E \sim L/\tau$ for some $\tau > 0$.\\

We consider the solutions $y_\gl^{(E)}$ of \eqref{Eq:yEz} starting from $(0,1)^\intercal$ at time $0$. For $\gl \in \R$, it is convenient to consider the associated polar coordinates, also called Pr\"ufer coordinates, implicitly defined by:
$$ (y^{(E)}_\gl)_1 = r^{(E)}_\gl(t) \sin \theta^{(E)}_\gl(t) \;,\quad (y^{(E)}_\gl)_2 = r^{(E)}_\gl(t) \cos \theta^{(E)}_\gl(t)\;.$$

Set $W^{(L)} := W_1^{(L)} + i W_2^{(L)} (:= \sqrt{2} \int_0^{\cdot} e^{i 2E' s} dB^{(L)}(s))$. Tedious applications of It\^o's formula (see also Remark \ref{Rk:Ito}) show that the equations for $r^{(E)}_\lambda$ and $\theta^{(E)}_\lambda$ are
\begin{equation}\label{Eq:SDEthetar}\begin{split}
d\theta^{(E)}_\lambda(t) &= \frac{\lambda}{2} dt - \frac{\sqrt{L/E}}{2} dB^{(L)}(t) + \frac{\sqrt{L/E}}{2\sqrt 2} \Re(e^{2i \theta^{(E)}_\lambda} dW^{(L)}(t)) + \cE(\theta^{(E)}_\lambda)(t)dt\;,\\
d\ln r^{(E)}_\lambda(t) &= \frac{L/E}8 dt + \frac{\sqrt{L/E}}{2\sqrt 2}  \Im(e^{2i \theta^{(E)}_\lambda} dW^{(L)}(t)) + \cE(\ln r^{(E)}_\lambda)(t)dt\;,
\end{split}\end{equation}
where the terms $\cE(\cdot)$, which will be proven to be negligible in the limit $L\to\infty$, are given by
\begin{align*}
\cE(\theta^{(E)}_\lambda)(t) &:= \frac{2\ell_E-\lambda}{2} \cos(2 \theta^{(E)}_\lambda(t) + 2E't) + \frac{L/E}4 \sin(2 \theta^{(E)}_\lambda(t)+2E't)\\
&\qquad- \frac{L/E}8  \sin(4 \theta^{(E)}_\lambda(t)+4E't) \;,\\
\cE(\ln r^{(E)}_\lambda)(t) &:=  \frac{2\ell_E - \gl}{2} \sin (2\theta^{(E)}_\lambda(t)+2E't)   - \frac{L/E}4 \cos(2 \theta^{(E)}_\lambda(t)+2E't)\\
&\qquad+ \frac{L/E}8 \cos(4 \theta^{(E)}_\lambda(t)+4E't)\;.
\end{align*}
In the above equations, the initial conditions are taken to be $\ln r^{(E)}_\lambda(0) = 0$ and $\theta^{(E)}_\lambda(0) = 0$.

\begin{remark}\label{Rk:Ito}
Recall $u_\gl^{(E)}$ from \eqref{Eq:uE} and note that $u_\gl^{(E)} = r^{(E)}_\gl \sin (\theta^{(E)}_\gl + E't)$ and $\frac1{L\sqrt E} (u_\gl^{(E)})' = r^{(E)}_\gl \cos (\theta^{(E)}_\gl + E't)$. Since the evolution equation of $u_\gl^{(E)}$ is simpler than that of $y_\gl^{(E)}$, one may prefer to apply It\^o's formula at this level, namely
$$ \ln r^{(E)}_\gl = \frac12 \ln ((u_\gl^{(E)})^2 + (\frac1{L\sqrt E} (u_\gl^{(E)})')^2)\;,\quad \theta^{(E)}_\gl + E't = \arccotan \frac{(u_\gl^{(E)})'}{L\sqrt E \,u_\gl^{(E)}}\;.$$
\end{remark}

Let $\Theta_{\lambda}, \Gamma_{\lambda}$ be the solutions of \eqref{Eq:Theta} starting from $\Theta_{\lambda}(0) = 0$ and $\ln\Gamma_{\lambda}(0) = 0$. Observe the similarity between the SDEs solved by $(\Theta_{\lambda}, \Gamma_{\lambda})$ and $(\theta^{(E)}_\gl,r^{(E)}_\gl)$.

\begin{proposition}\label{Prop:CVSDEs}
Fix $\tau > 0$ and consider $E=E(L) \sim L/\tau$. The collection, indexed by $L>1$, of continuous processes $(\theta^{(E)}_\gl(t),r^{(E)}_\gl(t) ; t\in [0,1], \gl \in \R)$ converges in law to $(\Theta_{\gl}(t),\Gamma_{\gl}(t) ; t\in [0,1], \gl \in\R)$, for the topology of uniform convergence on compact sets of $[0,1]\times \R$.
\end{proposition}

With this proposition at hand, we can proceed with the proof of the theorems.

\begin{proof}[Proofs of Theorems \ref{Th:Joint1} and \ref{Th:Joint2}]
In the proof of Theorem \ref{Th:Intensity} we saw that almost surely $\R \ni \gl \mapsto \Theta_\gl(1)$ is a continuous, increasing bijection from $\R$ to $\R$. The very same arguments ensure that this property also holds for $\R \ni \gl \mapsto \theta^{(E)}_\gl(1)$.\\
The convergence in law stated in Proposition \ref{Prop:CVSDEs} thus implies that the ordered sequence of hitting ``times'' of $\pi\Z$ by $\gl \mapsto \theta^{(E)}_{\gl}(1)$ converges to the corresponding sequence associated to $\gl \mapsto \Theta_{\gl}(1)$.\\
Note that $y_\gl^{(E)}$, resp.~$y_\gl$, is a continuous function of $(\theta^{(E)}_\gl,r^{(E)}_\gl)$, resp.~$(\Theta_{\gl},\Gamma_{\gl})$. We deduce that the point process
$$ \Big\{ \Big(\lambda, \frac{y_{\lambda}^{(E)}}{\|y_{\lambda}^{(E)}\|_{L^2}}\Big): \theta^{(E)}_{\lambda}(1) \in \pi\Z\Big\}\;,$$
converges in law to the point process
$$ \Big\{ \Big(\lambda, \frac{y_{\lambda}}{\|y_{\lambda}\|_{L^2}}\Big): \Theta_{\lambda}(1) \in \pi\Z\Big\}\;.$$
This is exactly the convergence stated in Theorem \ref{Th:Joint2}. Since the point processes involved in the convergence stated in Theorem \ref{Th:Joint1} are continuous projections of the above point processes, Theorem \ref{Th:Joint1} follows.
\end{proof}
\begin{remark}
As already mentioned, the convergence of the eigenvalues of Theorems \ref{Th:Joint1} and \ref{Th:Joint2} is the continuous analog of the result of \cite{KVV}. We believe that in Corollary 4 in \cite{KVV}, there should be no constant $\pi$ i.e. that the correct statement for the convergence of the eigenvalues of the discrete model (using their notations) is:
\begin{align*}
\Lambda_n - \arg(z^{2n+2}) \to \mbox{Sch}_\tau\,.
\end{align*}
\end{remark}

The next three subsections are devoted to the proof of Proposition \ref{Prop:CVSDEs}, while the last subsection provides the arguments for the proof of Proposition \ref{Prop:CVyv}.

\subsection{Tightness}\label{Subsec:Tightness}

Suppose we can show that for any $p\ge 2$, there exists a constant $C>0$ such that for all $L>1$, for all $\mu<\lambda$ and all $0 \le s \le t \le 1$
\begin{equation}\label{Eq:MomentsTime} \E[ |\theta^{(E)}_\mu(t) - \theta^{(E)}_\mu(s)|^{2p}] \le C |t-s|^p\;,\quad  \E[ |\ln r^{(E)}_\mu(t) - \ln r^{(E)}_\mu(s)|^{2p}] \le C |t-s|^p\;,\end{equation}
and
\begin{equation}\label{Eq:MomentsFreq} \E[ |\theta^{(E)}_\lambda(t) - \theta^{(E)}_\mu(t)|^{p}] \le C |\lambda-\mu|^p\;,\quad  \E[ |\ln r^{(E)}_\lambda(t) - \ln r^{(E)}_\mu(t)|^{p}] \le C |\lambda-\mu|^p\;.\end{equation}
Then, by Kolmogorov-Centsov's Theorem~\cite[Th 2.23 \& Th 14.9]{Kallenberg}, we deduce that there exists a constant $\beta > 0$ such that for any $p\ge 1$ and any compact set $K\subset \R$
$$ \sup_{L>1}\E\bigg[ \sup_{0\le s \le t \le 1} \sup_{\lambda,\mu \in K} \Big(\frac{|\theta^{(E)}_\mu(t) - \theta^{(E)}_\lambda(s)|}{(|t-s| + |\lambda-\mu|)^\beta}\Big)^p\bigg] < \infty\;,$$
and similarly for $\ln r^{(E)}_\gl$. Since in addition $\theta^{(E)}_\gl(0) = \ln r^{(E)}_\gl(0) = 0$, we deduce that the collection of processes is tight.\\

It remains to prove the above bounds. The increments in $t$ are easy to control: since the drift and diffusion coefficients of the SDE are bounded by some constant (uniformly over all parameters), we get the desired bound using the triangle inequality (to control separately the terms coming from the drift and the martingale) and the Burkholder-Davis-Gundy inequality (to control the martingale term). Note that the bound of the drift term is of order $|t-s|^{2p}$ while the bound of the martingale term is only of order $|t-s|^p$.\\
 
On the other hand, the increments in $\gl$ require some work: fix $p \geq 2$ and let us start with $\theta^{(E)}_\gl$. Since the coefficients of the SDE are Lipschitz in $\theta_\gl^{(E)}$, we deduce that there exists $C=C(p)>0$ such that for all $\mu \le \lambda$ and for all $t\in [0,1]$ we have
$$ \E[ |\theta^{(E)}_\lambda(t) - \theta^{(E)}_\mu(t)|^{p}] \le C\Big( |\lambda-\mu|^{p} + \int_0^t \E[|\theta^{(E)}_\lambda(s) - \theta^{(E)}_\mu(s)|^{p}] ds + \E[|M_t|^{p}]\Big)\;,$$
where
$$ M_t := \frac{\sqrt{L/E}}{2} \int_0^t \Big(\cos(2 \theta^{(E)}_\lambda(s)+2E's)-\cos(2 \theta^{(E)}_\mu(s)+2E's)\Big) dB^{(L)}(s)\;.$$
Combining the Burkholder-Davis-Gundy inequality and the Jensen inequality, there exists a constant $C'>0$ such that for all $t\in [0,1]$,
$$ \E[|M_t|^{p}] \le C' \int_0^t \E\Big[ |\theta^{(E)}_\lambda(s) - \theta^{(E)}_\mu(s)|^{p} \Big] ds\;.$$
The desired bound on $\E[ |\theta^{(E)}_\lambda(t) - \theta^{(E)}_\mu(t)|^{p}]$ then follows from Gr\"onwall's lemma.\\

We turn to $\ln r^{(E)}_\gl$. The strategy is the same, the only difference is that the coefficients of the SDE do not depend on $\ln r^{(E)}_\mu$ but on $\theta^{(E)}_\mu$. Since we already established bounds on the increments of the latter, one can easily conclude.

\subsection{Control of the error terms}\label{Subsec:Error}

Before we identify the limit of any converging subsequence, let us control the error terms appearing in the SDEs \eqref{Eq:SDEthetar}.
\begin{lemma}\label{lem:Errors}
For any $\gl\in\R$, the following convergences hold in probability as $L\to\infty$
\begin{equation*}\begin{split}
\sup_{0\le t \le 1} \Big| \int_0^t \cE(\theta^{(E)}_\gl)(s) ds\Big| \to 0\;,\qquad\sup_{0\le t \le 1} \Big| \int_0^t \cE(\ln r^{(E)}_\gl)(s) ds\Big| \to 0\;.
\end{split}
\end{equation*}
\end{lemma}
\begin{proof}
Given the terms that appear in $\cE$, it suffices to show that for any functions $f,g$ of the form $\cos(a \cdot), \sin(a \cdot)$, with $a\in \{1,2,4\}$, we have the following convergence in probability as $L\to\infty$
$$ \sup_{0\le t \le 1} \Big| \int_0^t f(\theta^{(E)}_\gl(s)) g(E's) ds \Big| \to 0\;.$$
This is not a direct consequence of the Riemann-Lebesgue Lemma since $\theta^{(E)}_\gl$ \emph{depends} on $L$. Without loss of generality, we can take $f=\sin(\cdot)$ and $g=\cos(\cdot)$. By It\^o's formula we find
\begin{align*}
\int_0^t \sin(\theta^{(E)}_\gl(s)) \cos(E's) ds &= \frac1{E'}\sin(\theta^{(E)}_\gl(t)) \sin(E't)\\
&- \frac1{E'} \int_0^t \sin(E's) \Big(\cos(\theta^{(E)}_\gl(s)) d\theta^{(E)}_\gl(s) - \frac12 \sin(\theta^{(E)}_\gl(s)) d\langle \theta^{(E)}_\gl\rangle_s\Big)
\end{align*}
Recall that $E'\to\infty$ as $L\to\infty$. Obviously, the first term on the r.h.s.~goes to $0$ uniformly over $t\in [0,1]$. Regarding the second term, it can be split into martingale and non-martingale terms. The non-martingale terms go to $0$ in probability uniformly over $t\in [0,1]$ since all the terms appearing inside the integral are uniformly bounded by some deterministic constant. The martingale term is given by
$$ N_t := -\frac1{E'} \int_0^t \sin(E's) \cos(\theta^{(E)}_\gl(s)) \frac{\sqrt{L/E}}{2\sqrt 2} \Re(e^{2i \theta^{(E)}_\gl} dW^{(L)}(s))\;.$$
The Burkholder-Davis-Gundy inequality ensures that there exists a constant $C>0$ such that
$$ \E[\sup_{0\le t \le 1} N_t^2] \le C \E[\langle N\rangle_1]\;.$$
Since the r.h.s.~is of order $(1/E')^2$, we deduce that $\sup_{0\le t \le 1} |N_t|$ goes to $0$ in probability, as required.
\end{proof}

\subsection{Identification of the limit}\label{Subsec:IdLim}

Fix $\gl_1, \ldots, \gl_n \in \R$. We will identify the law of any converging subsequence of $(\theta_{\gl_i}^{(E)},\ln r_{\gl_i}^{(E)})_{1\le i \le n}$ through the following standard martingale problem, whose proof can be found in~\cite[Prop 5.4.6]{Karatzas}

\begin{proposition}
Let $(\theta_{\gl_i},\ln r_{\gl_i})_{1\le i \le n}$ be a continuous process on $[0,1]$ and let $\cF$ be the associated filtration. Assume that
\begin{equation*}\begin{split}
M_{\gl_i}(t) = \theta_{\gl_i}(t) - \frac{\gl_i}{2} t\;,\quad N_{\mu_i}(t) = \ln r_{\mu_i}(t) - \frac{\tau}{8}t\;.
\end{split}\end{equation*}
together with
\begin{equation}\label{Eq:Brackets2}\begin{split}
&M_{\gl_i}(t)M_{\gl_j}(t) - \frac{\tau}{4} t - \frac{\tau}{8} \int_0^t \cos(2\theta_{\gl_i} - 2\theta_{\gl_j}) ds\;,\\
&N_{\gl_i}(t)N_{\gl_j}(t) - \frac{\tau}{8} \int_0^t \cos(2\theta_{\gl_i} - 2\theta_{\gl_j}) ds\;,\\
&M_{\gl_i}(t)N_{\gl_j}(t) - \frac{\tau}{8} \int_0^t \sin(2\theta_{\gl_j} - 2\theta_{\gl_i}) ds\;,
\end{split}\end{equation}
are $\cF$-martingales. Then $(\theta_{\gl_i},\ln r_{\gl_i})_{1\le i \le n}$ coincides in law with the unique solution of the SDEs \eqref{Eq:Theta} associated to the parameters $\gl_1,\ldots,\gl_n$.
\end{proposition}

Recall the SDEs \eqref{Eq:SDEthetar} and define the martingales
\begin{equation}\label{Eq:Mgale}\begin{split}
M_{\gl_i}^{(E)}(t) &= \theta^{(E)}_{\gl_i}(t) - \frac{\gl_i}{2} t - \int_0^t \cE(\theta^{(E)}_{\gl_i})(s) ds\;,\\
N_{\gl_i}^{(E)}(t) &= \ln r^{(E)}_{\gl_i}(t) - \frac{L/E}{8} t - \int_0^t \cE(\ln r^{(E)}_{\gl_i})(s) ds\;.
\end{split}\end{equation}
From the moment bounds established for tightness, we easily deduce that all moments of these martingales are bounded uniformly over $L> 1$.\\

Let $(\theta_{\gl_i},\ln r_{\gl_i})_{1\le i \le n}$ be the limit of a converging subsequence of $(\theta_{\gl_i}^{(E)},\ln r_{\gl_i}^{(E)})_{1\le i \le n}$: for simplicity we keep the same notation for the subsequence. We naturally define
\begin{equation}\label{Eq:Mgale2}\begin{split}
M_{\gl_i}(t) &= \theta_{\gl_i}(t) - \frac{\gl_i}{2} t\;,\\
N_{\gl_i}(t) &= \ln r_{\gl_i}(t) - \frac{\tau}{8}t\;.
\end{split}\end{equation}
Thanks to Lemma \ref{lem:Errors}, we can pass to the limit on \eqref{Eq:Mgale} and we obtain \eqref{Eq:Mgale2}. Given the aforementioned moment bounds, we also deduce that $M_{\gl_i}, N_{\gl_i}$ are martingales (in the natural filtration associated to the processes at stake). We now identify their brackets.\\

Again from \eqref{Eq:SDEthetar} we see that the processes
\begin{equation}\label{Eq:Brackets}\begin{split}
&M_{\gl_i}^{(E)}(t)M_{\gl_j}^{(E)}(t) - \frac{L/E}{4} \int_0^t (-1+ \cos(2\theta^{(E)}_{\gl_i}(s)+2E's)) (-1+ \cos(2\theta^{(E)}_{\gl_j}(s)+2E' s))  ds\;,\\
&N_{\gl_i}^{(E)}(t)N_{\gl_j}^{(E)}(t) - \frac{L/E}{4} \int_0^t \sin(2\theta^{(E)}_{\gl_i}(s)+2E' s) \sin(2\theta^{(E)}_{\gl_j}(s)+2E' s) ds\;,\\
&M_{\gl_i}^{(E)}(t)N_{\gl_j}^{(E)}(t) - \frac{L/E}{4} \int_0^t (-1+ \cos(2\theta^{(E)}_{\gl_i}(s)+2E' s)) \sin(2\theta^{(E)}_{\gl_j}(s)+2E' s) ds\;,
\end{split}\end{equation}
are martingales. We aim at passing to the limit on \eqref{Eq:Brackets}. We can compute the limits of the three integrals therein: by expanding the $\cos$ and $\sin$ functions, the oscillating terms in $E'$ will vanish thanks to the Riemann-Lebesgue-type argument of the previous subsection, and the remaining terms match with the ones in the integrals that appear in \eqref{Eq:Brackets2}. Combining this with the aforementioned moment bounds, we deduce that the processes of \eqref{Eq:Brackets2} are also martingales. We can then apply the martingale problem recalled above and this completes the proof of Proposition \ref{Prop:CVSDEs}.

\subsection{Proof of Proposition \ref{Prop:CVyv}}
The proof is very close to the proof of Proposition \ref{Prop:CVSDEs}. The main difference is that we consider solutions of \eqref{Eq:yEz} with a non-real parameter $z$ so that the polar representation used previously does not hold anymore: hence we work directly at the level of the SDEs \eqref{Eq:yEz} to prove the convergences. Note that we prove the convergence for a fixed $z$, although the arguments could be adapted to get the local uniform convergence of the family.\\

The SDE \eqref{Eq:yEz} solved by $y^{(E)}_z$ can be written in the following way:
\begin{align*}
dy^{(E)}_z(t) &= \Big(\frac{z}{2} J^{-1} dt - \frac{\sqrt{L/E}}{2}J^{-1}\begin{pmatrix} d{B}^{(L)}(t) + \frac1{\sqrt{2}}dW_1^{(L)}(t) & \frac1{\sqrt{2}}dW_2^{(L)}(t) \\ \frac1{\sqrt{2}}dW_2^{(L)}(t) & dB^{(L)}(t) - \frac1{\sqrt{2}}dW_1^{(L)}(t) \end{pmatrix}\Big) y^{(E)}_z(t)\\
&\qquad+ \cE(t) y^{(E)}_z(t) dt\;,
\end{align*}
where
$$ \cE(t) := \frac{z-2\ell_E}2 J^{-1}(2R_{E'}-I)\;,\quad  J^{-1}(2R_{E'} - I) = \begin{pmatrix} \sin 2E't  & -\cos 2E't \\ -\cos 2E't & -\sin 2E't \end{pmatrix}\;.$$

The putative limit satisfies
$$dy_z(t) = \Big(\frac{z}{2} J^{-1} dt - \frac{\sqrt{\tau}}{2}J^{-1}\begin{pmatrix} d\cB(t) + \frac1{\sqrt{2}} d\cW_1(t) & \frac1{\sqrt{2}}d\cW_2(t) \\ \frac1{\sqrt{2}}d\cW_2(t) & d\cB(t) - \frac1{\sqrt{2}} d\cW_1(t) \end{pmatrix}\Big) y_z(t)\;.$$

The processes $y^{(E)}_z$ and $v^{(E)}_z$, resp.~$y_z$ and $v_z$, satisfy the same equations, the only difference lies in the initial conditions:
$$ y^{(E)}_z(0) = y_z(0) = \begin{pmatrix} 0\\ 1 \end{pmatrix}\;,\quad v^{(E)}_z(0) = v_z(0) = \begin{pmatrix} 1\\ 0 \end{pmatrix}\;.$$
Consequently, the proof of the tightness relies on exactly the same arguments for $y^{(E)}_z$ and $v^{(E)}_z$, and we restrict ourselves to presenting the details for the former. We start with some a priori bounds.

\begin{lemma}\label{Lemma:Moments}
For any $p \geq 1$, there exists $C>0$ such that
$$\sup_{L>1} \sup_{t\in [0,1]}\E[|y^{(E)}_z(t)|^p] < \infty\;,\quad \sup_{t\in [0,1]}\E[|y_z(t)|^p] < \infty\;.$$
\end{lemma}
\begin{proof}
Fix $p \geq 2$. From the integral form of the SDE above, applying successively the Burkholder-Davis-Gundy inequality and the Jensen inequality we get the existence of some deterministic constants $C,C'>0$ (depending on $z$ and $p$) such that uniformly over all $L$ and $t$:
\begin{align*}
\E[|y^{(E)}_z(t)|^{p}] &\le 1 + C \Big(\int_0^t \E[|y^{(E)}_z(s)|^{p}] ds + \E\big[\big(\int_0^t |y^{(E)}_z(s)|^2 ds\big)^{p/2}\big]\Big)\\
&\le 1 + C' \int_0^t \E[|y^{(E)}_z(s)|^{p}] ds \;.
\end{align*}
Gr\"onwall's Lemma then yields the desired bound. The proof is the same for $y_z$.
\end{proof}

We now control the oscillations thanks to the following lemma, which is also used in Subsection \ref{subsec:StrongResolventCV} for the proof of the strong resolvent convergence.
\begin{lemma}\label{Lemma:RL}
Let $h:[0,1]\to\R$ be a smooth function and let $f$ be either $\sin(2E'\cdot)$ or $\cos(2E'\cdot)$. Then for any $i\in\{1,2\}$ we have as $L\to\infty$
$$ \E\Big[\sup_{0\le t \le 1} \Big| \int_0^t f(s) (y^{(E)}_z)_i(s) h(s) ds \Big|\Big] \to 0\;.$$
\end{lemma}
\begin{proof}
The arguments are essentially the same as those of the proof of Lemma \ref{lem:Errors}. Take $f=\cos(2E'\cdot)$ and $i=1$ without loss of generality. By It\^o's formula, we have
\begin{align*}
\int_0^t \cos(2E' s) (y^{(E)}_z)_1(s) h(s) ds &= \frac1{2E'} (y^{(E)}_z)_1(t) h(t) \sin(2E' t)\\
&-\frac1{2E'} \int_0^t \sin(2E's) \Big(h'(s)(y^{(E)}_z)_1(s) + h(s) d(y^{(E)}_z)_1(s) \Big) ds\;.
\end{align*}
Since $E'\to\infty$ and given the bounds of Lemma \ref{Lemma:Moments} it is easy to check that the expectation of all terms goes to $0$, except for the martingale term produced by $d(y^{(E)}_z)_1(s)$ which requires some additional work. This martingale term is given by
$$ N_t := \frac{\sqrt{L/E}}{4E'} \int_0^t \sin(2E's) h(s) \Big(  \frac1{\sqrt 2} (y^{(E)}_z(s))_1 dW_2^{(L)}(s) +  (y^{(E)}_z(s))_2 (dB^{(L)}(s) - \frac1{\sqrt 2} dW_1^{(L)}(s)) \Big) ds\;.$$
By the Burkholder-Davis-Gundy inequality there exists $C>0$ such that
$$ \E\big[ \sup_{0\le t \le 1} N_t^2\big] \le \frac{C}{(E')^2} \E\Big[\int_0^t h(s)^2 |y^{(E)}_z(s)|^2 ds\Big]\;,$$
which, in view of Lemma \ref{Lemma:Moments}, goes to $0$ as $L\to\infty$.
\end{proof}

As a consequence of Lemma \ref{Lemma:RL}, we deduce that $\sup_{0\le t \le 1} \Big|\int_0^t \cE(s) y_z^{(E)}(s) ds\Big|$ goes to $0$ in probability as $L\to\infty$.\\

Let us now prove tightness of $(y^{(E)}_z)_{L>1}$ (once again, the arguments are exactly the same for $(v^{(E)}_z)_{L>1}$). Fix $p\ge 1$. By the triangle and the Burkholder-Davis-Gundy inequalities at the first line, the H\"older inequality at the second line and Lemma \ref{Lemma:Moments} at the third line, there exist some constants $C,C'>0$ such that for all $L>1$ and all $0 \le s \le t \le 1$
\begin{align*}
\E[|y^{(E)}_z(t) - y^{(E)}_z(s)|^{2p}] &\le C \Big( \E\Big[\Big(\int_s^t |y^{(E)}_z(r)| dr\Big)^{2p}\Big] + \E\Big[\Big(\int_s^t |y^{(E)}_z(r)|^2 dr\Big)^p\Big] \Big)\\
&\le C  \int_s^t \E\Big[|y^{(E)}_z(r)|^{2p}\Big] dr \times \Big((t-s)^{2p-1} + (t-s)^{p-1}\Big)  \\
&\le C' |t-s|^{p}\;.
\end{align*}
By Kolmogorov-Centsov's Theorem~\cite[Th 2.23 \& Th 14.9]{Kallenberg}, we deduce that there exists a constant $\beta > 0$ such that for any $p\ge 1$
$$ \sup_{L>1}\E\bigg[ \sup_{0\le s \le t \le 1} \Big(\frac{|y^{(E)}_z(t) - y^{(E)}_z(s)|}{|t-s|^\beta}\Big)^p\bigg] < \infty\;,$$
and tightness follows.\\

The identification of the limit of any converging subsequence of $(y^{(E)}_{z},v^{(E)}_{z})_{L>1}$ can be carried out with a martingale problem as in Subsection \ref{Subsec:IdLim} and with the help of Lemma \ref{Lemma:RL}: the arguments being virtually the same, we do not provide the details.

\section{Top of the spectrum}\label{sec:top}

In this section, we assume that $E = E(L)\gg L$ and we explain how the previous arguments can be adapted to establish Theorem \ref{Th:Top}. First of all, since the limiting objects appearing in that statement are all deterministic the asserted convergences in probability are granted provided that convergence in law holds, and this is what we are going to prove.\\

Let $\cD(\mathtt{F})$ be the closure in $H^1((0,1),\R^2)$ of all smooth functions $f:[0,1]\to\R^2$ such that $f_1(0) =f_1(1) = 0$. One can check that the operator $\mathtt{F}$ on $\cD(\mathtt{F})$ is self-adjoint. Following the same steps as in Section \ref{sec:limitingoperator} we see that the SDEs associated to the operator $\mathtt{F}$ are trivial:
\begin{align*}
&dy_z(t)  = -\frac{z}{2} J y_z(t) dt\;,\quad t \in [0,1]\;,\quad z \in \C\;,\quad y_z(0) = \begin{pmatrix} 0 \\ 1\end{pmatrix}\;,
\end{align*}
and for $\gl\in\R$
\begin{align*}
&d \Theta_{\gl}(t) = \frac{\gl}{2} dt \;, \quad\Theta_{\gl}(0) = 0\;,\\
&d \ln \Gamma_{\gl}(t) = 0 dt\;,\quad \Gamma_{\gl}(0) = 1\;.
\end{align*}
Their solutions are given by $y_\gl(t) = (\sin (\lambda t/2),\cos(\lambda t/2))$, $\Theta_\gl(t) = \lambda t/2$ and $\Gamma_\gl(t) = 1$.\\

The proofs of the convergences are \emph{exactly} the same as those presented in Sections \ref{sec:CVoperator} and \ref{Sec:SDEs}: the only difference is that many terms, that had non-trivial contributions in the limit in the regime $E\sim L/\tau$, now vanish in the limit since $L/E \to 0$.

\subsection*{Acknowledgements} The work of CL is supported by the project SINGULAR ANR-16-CE40-0020-01. The authors thank Gaultier Lambert for his useful comments on the first version of this paper.

\bibliographystyle{Martin}
\bibliography{library}

\end{document}